\documentclass{sn-jnl}
\usepackage{titlesec}
\titlespacing*{\section}{0pt}{*1.5}{*0.8}
\titlespacing*{\subsection}{0pt}{*1.2}{*0.6}
\titlespacing*{\subsubsection}{0pt}{*1}{*0.5}
\usepackage{amsthm}
\newtheoremstyle{tight}%
  {4pt}{4pt}{\itshape}{}{\bfseries}{.}{0.5em}{}
\theoremstyle{tight}

\usepackage{graphicx}%
\usepackage{subcaption}
\usepackage{multirow}%
\usepackage{float}%
\usepackage{amsmath,amssymb,amsfonts}%
\usepackage{mathtools}%
\usepackage{thmtools,thm-restate}
\makeatletter
\renewcommand\thmt@autorefsetup{\@xa\def\csname\thmt@envname autorefname\@xa\endcsname\@xa{\thmt@thmname}}
\makeatother
\makeatother
\usepackage{mathrsfs}%
\usepackage[title]{appendix}%
\usepackage{xspace}
\usepackage{xcolor}%
\usepackage{textcomp}%
\usepackage{manyfoot}%
\usepackage{booktabs}%
\usepackage{algorithm}%
\usepackage{algorithmicx}%
\usepackage{algpseudocode}%
\usepackage{listings}%
\usepackage[normalem]{ulem}%
\usepackage{tikz}%
\usepackage{stfloats}
\usepackage[makeroom]{cancel}
\usepackage[utf8]{inputenc}
\usepackage{hyperref}

\usepackage{doi}

\usetikzlibrary{matrix,arrows,decorations.pathmorphing}%

\newtheorem{theorem}{Theorem}
\newtheorem{conjecture}{Conjecture}
\newtheorem{proposition}[theorem]{Proposition}
\newtheorem{corollary}{Corollary}

\newtheorem{remark}{Remark}
\newtheorem{lemma}[theorem]{Lemma}

\newcommand\C{{\mathbb C}}

\newcommand{\AlphaEvolve}{\texttt{AlphaEvolve}\xspace}
\newcommand{\restr}[2]{%
  #1\!\vert_{#2}%
}
\def\sech{{\operatorname{sech}}}
\def\csch{{\operatorname{csch}}}
\def\cosh{{\operatorname{cosh}}}
\def\arccoth{{\operatorname{arccoth}}}

\newtheorem{result}{Result}

\def\oz1{d{\overline z}^1}
\def\oz2{d{\overline z}^2}
\def\oz3{d{\overline z}^3}

\def\oI{\overline I}

\def\oz{\overline z}

\def\oIq1{\oI_1\cdots\oI_{q-1}}
\def\oIq2{\oI_1\cdots\oI_{q-2}}

\def\supp{\mbox supp}

\begin{document}

\title[Article Title]{Smale's Mean Value Conjecture and its Dual Conjecture for Complex Polynomials}

\author[1]{\fnm{Aneesh} \sur{Jatar}}\email{aneesh.jatar1@gmail.com}
   
\author[2]{\fnm{Tuen-Wai} \sur{ Ng}}\email{ntw@maths.hku.hk}

\affil[1,2]{\orgdiv{Department of Mathematics}, \orgname{The University of Hong Kong}, \orgaddress{\street{Pokfulam Road}, \city{Hong Kong}, \country{China}}}

\abstract{For the purposes of proving Smale’s mean value conjecture, one may restrict consideration to an explicitly described family of so-called {\it Schlicht normalized polynomials}. For Schlicht normalized polynomials whose leading coefficient decays sufficiently slowly, we show that the conjectured upper bound $1$ becomes asymptotically valid as the degree $d$ tends to infinity. The key and novel input is a refined Koebe $\frac{1}{4}$-type theorem of Cunningham, tailored to Schlicht functions whose images have bounded logarithmic capacity. For the dual mean value conjecture, we obtain an improvement of Eremenko’s Markov-type inequality for regions bounded by polynomial lemniscates on which the polynomial is univalent. As a consequence, we improve the best known unconditional lower bound $\frac{1}{d^2}$ of Dubinin to $\frac{(d - \frac{1}{2})^{1/d}}{d^2}$ for all $d \geq 2$. The strengthened Markov-type inequality constitutes one of the main technical contributions of this paper. Its proof combines the solution to a related extremal problem for the logarithmic capacity of polynomial lemniscates, with the quasiconformal deformation method introduced by Eremenko and Hayman to establish the connectedness of extremal polynomial lemniscates arising in the Erd\H{o}s–Herzog–Piranian problem on maximal lemniscate length. }

\maketitle

\section{Introduction and statement of results}\label{sec1}

\subsection{Mean and Dual Mean Value Conjecture}
 
 In Stephen Smale's seminal paper \cite{Smale1981} on the efficiency and complexity of Newton's method for approximating zeros of a non-constant complex polynomial $P$ of one complex variable, Smale proposed to study the smallest positive constant $c$ such that for any point $a$ in the complex plane $\mathbb{C}$ with $P'(a)\neq 0$, $\left|\frac{P(b)-P(a)}{b-a}\right| \le c|P'(a)|$ for at least one critical point $b$ of $P$ (a zero of $P'$).  
Let $M$ be  the least possible value of the factor $c$ for all non-linear polynomials and $M_d$ be the corresponding value for polynomials of degree $d$. It was proven by Smale \cite{Smale1981} that $1 \le M\le4$ and he conjectured that $M=1$ or even $M_d=\frac{d-1}{d}$ and pointed out that the number $d-1 \over d$ would, if true, be the best possible bound here as it is attained (for any nonzero $\lambda$) when $P(z) = z^d-\lambda z$ and $a=0$. The conjecture is now known as Smale's mean value conjecture which is still unsolved and is also listed as one of the three minor problems in Smale's famous problem list \cite{Smale2000}. It was also included as Problem 6.23 and studied by the AI tool \AlphaEvolve in Georgiev, G\'omez-Serrano, Tao and Wagner's problem list \cite{GGSJTW2025} recently, where the authors note that the AI tool was unable to find a counterexample to the above conjectured value of $M_d$.  \\

To prove Smale's mean value conjecture, it suffices to consider polynomials with certain normalizations as follows. Let $Z(P')$ denote the set of critical points of $P$. For each critical point $b \in Z(P')$ and any
 noncritical point $a$ of $P$, define 
$$S(P,a,b)=\frac{P(a)-P(b)}{(a-b)P'(a)}.$$
Let $L$ be a non-constant linear polynomial. It is easy to check that $S(P \circ L^{-1}, L(a),L(b))=S(P,a,b)$ for each critical point $b \in Z(P')$. For a non-linear polynomial $P$ fixing the origin, with $P'(0) \neq 0$ and $b$ a critical point of $P$, we define $$S(P,b) \coloneq |\frac{P(b)}{b P'(0)}|$$ It follows that to prove Smale's mean value conjecture, one needs to only consider the case $a=0$ because it is always possible to choose a suitable linear polynomial $L$ such that $L(a)=0$. Thus Smale's mean value conjecture is equivalent to the following normalized conjecture:\\

\noindent
\begin{conjecture}
Let $P$ be a polynomial of degree $d \ge 2$ such that $P(0)=0$ and $P'(0) \neq 0$.  Then

$$S(P) \coloneq \min_{b \in Z(P')} S(P,b)  \le {d-1 \over d} < 1\ $$\\

\end{conjecture}

A natural ``dual'' version of the above conjecture was posed by Dubinin and Sugawa \cite{dubinin2009dualmeanvalueproblem} and independently Ng \cite{Ng}, in 2009:

\begin{conjecture}
 Let $P$ be a polynomial of degree $d \ge 2$ such that $P(0)=0$ and $P'(0) \neq 0$. Then

$$D(P) \coloneq \max_{b \in Z(P')} S(P,b)  \geq {1 \over d} \,  $$\\   
\end{conjecture}

Dubinin and Sugawa \cite{dubinin2009dualmeanvalueproblem} also showed that $D(P)\ge \frac{1}{d4^d}$ which was then slightly improved to $\frac{1}{4^d}$ in \cite{NZ2016}. The conjectured lower bound $\frac{1}{d}$, if true, would be best possible as it is attained for the polynomials $P(z) = (z+1)^d - 1$. Henceforth we will only consider nonlinear polynomials fixing the origin, with nonzero derivative at the origin.

To date, the best known unrestricted (without further assumptions on the polynomial $P$ besides its degree ) upper bound for the quantity $$S(P) \coloneq \min_{b \in Z(P')}\left|{P(b) \over bP'(0)}\right|$$ holding for polynomials of degree $d \geq 2$ is asymptotically of the form $4 - \mathcal{O}(\frac{1}{\sqrt{d}})$ as $d \rightarrow \infty$, where $d$ is the degree of $P$. This bound was established by Crane in 2007 \cite{Crane2007} using properties of special extremal polynomials for the mean value conjecture (referred to by Crane as ``Standard extremal polynomials''), whose existence was established first by Crane \cite{Crane2006} and Ng \cite{Ng} later by the theory of amoeba. The first two theorems of this paper show that, in certain asymptotic regimes as $d \rightarrow \infty$, the asymptotic constant of $4$ may be replaced by $3$ or even $1$. Notice that all the previous estimates \cite{BMN2002,CFL2007,FS2006,Crane2007} applied to this scenario can only produce the constant $4$.

The best known unrestricted (without further assumptions on the polynomial $P$ besides its degree) lower bound for the quantity $$D(P) = \max_{b \in Z(P')} |\frac{P(b)}{b P'(0)}|$$ is $\frac{1}{d^2}$, where $d$ is the degree of $P$. This lower bound follows from Theorem 2 of the 2010 paper of Dubinin $\cite{Dubinin2010}$, which establishes a more general and in fact, $\textbf{sharp}$ (i.e. best possible) finite-increment theorem for complex polynomials. 

Related to the dual mean value problem is the problem of choosing a critical point $b \in Z(P')$ which satisfies $S(P,b) \geq \frac{1}{d}$. To this end it was proved in \cite{Hinkkanen2025DualSM} that for polynomials $P$ of degree $7$, there exists a critical point $b$ of minimal modulus for which $S(P,b) \geq \frac{1}{7}$ holds. In fact, by an argument involving the Walsh coincidence theorem, one can show that if $P$ is a non-linear polynomial of degree less than or equal to $7$, given any critical point $b \in Z(P')$ of minimal modulus, one has $$S(P,b) \geq \frac{1}{d}$$ where $d = \deg(P)$. On the other hand, for higher degree polynomials, this strategy is not effective. In particular, for the degree $9$ polynomial $P_9(z) = \int_{0}^{z} a^{-7} (t-1) (t-a)^7 dt$, where $a = \exp(i\frac{18 \pi}{41})$ we have that $P_9$ has exactly two critical points, one of multiplicity 7 at $z=a$ and a simple critical point at $z=1$. Thus all critical points have minimal modulus in this case, but for the critical point $z=1$ we have $$S(P_9,1) = |P_9(1)| \approx 0.00169... < \frac{1}{9^2} = 0.01234...$$

The above example shows that in general, an arbitrary critical point of minimal modulus will not do better than the best unconditional lower bound for the dual mean value problem of $\frac{1}{d^2}$. Indeed, in the case of $P_9$, it is necessary to choose the critical point of minimal modulus and highest multiplicity for this purpose. In the case of $P_9$, the ``correct'' choice of critical point is also the critical point corresponding to the critical value of largest modulus. 

On the other hand, Dubinin has shown in \cite{Dubinin2019} that if one chooses an arbitrary critical point corresponding to a  critical value of largest modulus, one cannot do better than $\frac{\pi}{4d^2}$ asymptotically as $d \rightarrow \infty$ for the dual mean value lower bound. It is thus natural to ask whether there is a subset of critical points for which, one can obtain a lower bound strictly larger than $\frac{1}{d^2}$ for each $d \geq 2$. Corollary \ref{cor:2} of this paper shows that choosing any critical point $b$ lying on the boundary of the maximal univalent lemniscate $D(P,t(P))$ will always serve for this purpose (although there is no asymptotic improvement as $d \rightarrow \infty$). Corollary \ref{SharperDualIneq} of this paper shows that, for the purposes of maximizing $S(P,b)$ amongst the critical points $b \in Z(P') \cap \partial D(P,t(P))$, one should choose the critical point of $b \in \partial D(P,t(P)) \cap Z(P')$ for which the image of the segment $$[0,P(b)] \coloneq \{t P(b) : t \in [0,1] \}$$ by the univalent branch $f$ of $P^{-1}$ fixing the origin, $\gamma = f([0,P(b)])$, deviates most from the line segment $[0,b] \coloneq \{ t b : 0 \leq t \leq 1 \}$.

\subsection{Statement of results and preliminary constructions} \label{preliminaryConstructionsandStatementOfResults}
In this paper, we establish three main results, the first two results establish conditional improvements on the asymptotic constant of $4$ for the Smale's mean value conjecture (see Theorems \ref{thm:1},\ref{thm:2} and Corollary \ref{cor:1}) in certain asymptotic regimes as $d \rightarrow \infty$ (namely for Schlicht normalized polynomials whose leading coefficient decays sufficiently slowly), whereas the third result (Theorem \ref{thm:3}) provides an improved estimate of Markov type for the so-called \textbf{maximal univalent lemniscates}, which are the maximal lemniscate regions which the polynomial maps univalently onto a disk. Theorem \ref{thm:3} provides an improvement, for these regions, of the Markov type inequality of Eremenko (Theorem 1 of \cite{Eremenko2007}).

As a consequence of this improved Markov type inequality, we are able to establish an improvement on the best known lower bound for the \textbf{dual} mean value conjecture (see Corollaries \ref{cor:2} and \ref{SharperDualIneq}).

Before, we state our main results, we introduce a special normalization for the polynomials in Conjectures 1 and 2 and define the so-called \textbf{maximal univalent lemniscate} $D(P,t(P))$, of a polynomial $P$, which is the principal object in all the main results of this paper. Let $P$ be a non-linear polynomial of degree $d$, fixing the origin, with non-vanishing derivative at the origin. For any $s > 0$ we let $$D(P,s)$$ denote the component of $P^{-1}(\mathbb{D}_{s})$ containing the origin, where $$\mathbb{D}_{s} \coloneq \{z \in \mathbb{C}: |z| < s \}$$ Since $P$ has non-vanishing derivative at the origin, for sufficiently small $s > 0$ $$\restr{P\ }{D(P,s)} : D(P,s) \rightarrow \mathbb{D}_{s}$$ is a conformal bijection. We can then define $$t(P) \coloneq \sup \{ s : \restr{P\ }{D(P,s)} : D(P,s) \rightarrow \mathbb{D}_{s} \text{ is a conformal bijection} \} $$ Since $P$ is non-linear, we necessarily have that $$t(P) < \infty$$ moreover $$\restr{P \ }{D(P,t(P))} : D(P,t(P)) \rightarrow \mathbb{D}_{t(P)}$$ is a conformal bijection, so that $$t(P) \in \{ s : \restr{P\ }{D(P,s)}: D(P,s) \rightarrow \mathbb{D}_{s} \text{ is a conformal bijection} \} $$ We now show that for any $0<s\leq t(P)$, $D(P,s)$ is a Jordan domain. To see this note that the part of the boundary of $D(P,s)$ near a boundary point $z \in \partial D(P,s)$ must (locally) consist of at most $2d$ real-analytic Jordan arcs (and only a single real-analytic Jordan arc centered at $z$, if $z$ is a boundary point which is not a critical point of $P$) issuing from $z$ at equal angles, and forming part of the lemniscate $\{z \in \mathbb{C}: |P(z)| = s \}$, so that we conclude $$\partial D(P,s)$$ is locally-connected. Thus, the inverse branch $f$ of $P^{-1}$ fixing the origin defined on $\mathbb{D}_{s}$, must extend continuously to $\overline{\mathbb{D}_{s}}$ by Carath\'eodory's theorem (see Theorem 2.6 of page 24 of \cite{pommerenke1992boundary})), and by continuity, we necessarily have $$P(f(z)) = z$$ for all $z \in \overline{\mathbb{D}_{s}}$. It follows that $f$ extends to a homeomorphism of $\overline{\mathbb{D}_{s}}$ onto $\overline{D(P,s)}$ so that $\partial D(P,s) $ is a Jordan curve. By definition of $t(P)$, there is necessarily a critical point $b \in Z(P')$ lying on the (Jordan) boundary of $\partial D(P,t(P))$ (if not, we could extend $P$ to a conformal bijection of $D(P,s)$ onto $\mathbb{D}_{s}$ for some $s > t(P)$). It is thus equivalent to define $t(P)$ to be the maximum of the quantities $$|P(b)|$$ taken over the set of critical points $b$ of $P$ for which $P$ is a conformal bijection of $D(P,|P(b)|)$ onto $\mathbb{D}_{|P(b)|}$. In view of the above definition, we call $D(P,t(P))$ the \textbf{maximal univalent lemniscate} of $P$ (with respect to the origin). For any critical point $b \in Z(P')$, we define $$S(P,b) \coloneq |\frac{P(b)}{P'(0) b}|$$

Note that for any non-zero complex numbers $\alpha$ and $\beta$, we have for $Q(z):=\alpha P (\beta z)$, that  $S(Q,\frac{b}{\beta})=S(P,b)$ for any critical point $b \in Z(P')$. It follows that we may choose appropriate $\alpha$ and $\beta$ so that $t(Q)=1$ and $Q'(0)=1$. We call $Q$ the \textbf{Schlicht normalization} of the polynomial $P$. So for Conjecture 1 and 2, we can always (without loss of generality) assume  that 
\begin{equation}
P'(0)=1=t(P) 
\end{equation}
A polynomial satisfying these two conditions plus the condition that $P(0)=0$ will be referred to as \textbf{Schlicht normalized}. If $P$ is a Schlicht normalized polynomial, then the branch of   $P^{-1}:\mathbb{D} \to D(P,1)$ fixing the origin is a Schlicht function (in the usual sense of univalent function theory), which extends continuously to a homeomorphism $P^{-1} : \overline{\mathbb{D}} \to \overline{D(P,1)}$, and by the above discussion, $\partial D(P,1)$ is a (piecewise-analytic) Jordan curve containing at least one critical point of $P$. For $c>0$ and $d\ge 2$, we let $$\mathcal{P}_{d,c}:=\{P(z)=a_dz^d +\cdots+z: t(P)=1, |a_d|=c\}.$$

\medskip
Notice that by P\'olya's theorem (Proposition \ref{Polya}), for $\mathcal{P}_{d,c}$ to be non-empty, we necessarily have $0<c <1$. In fact, as a consequence of Theorem 2 in \cite{dubinin2012some}, the classes $\mathcal{P}_{d,c}$ are empty for $$c > \frac{1}{d}(1 - \frac{1}{d})^{d-1} = c(Q^{\ast})\sim \frac{1}{ed}$$ where $$Q^{\ast}$$ is the Schlicht normalization of $$P^{\ast}(z) \coloneq z^d - z$$ and $$c(Q^{\ast})$$ is the modulus of the leading coefficient of $Q^{\ast}$\footnote{One can also show (though this will not be needed anywhere in the paper) that the parameter range $$c \in (0, \frac{1}{d}(1 - \frac{1}{d})^{d-1}]$$ is exact, in the sense that for any $$c \in (0,  \frac{1}{d}(1 - \frac{1}{d})^{d-1}]$$ there exists a Schlicht normalized polynomial $P(z) = a_dz^d + ... + z$ of degree $d$, with $|a_d| = c$. Since we will not need this result anywhere else in the paper we omit a proof of this result.}

We now state the main results of this paper.

\begin{theorem}\label{thm:1}
 Let $L_c := \ln(1/c)$ and $g(s) := (1 + s)^2 \cdot s^{\frac{-2s}{1+s}}$ for $s \in (0,1)$. Then for any $P\in \mathcal{P}_{d,c}$, and critical point $b \in Z(P') \cap \partial D(P,1)$

\begin{itemize}
\item[i)] $S(P,b) \leq g(s_{*})$, where $s_{*} = \sqrt{1 - W_0(e c^{1/(2d)})}$ (Lambert-W version). Here $W_0$ is the principal branch of the Lambert function $W$. 
\item[ii)] $S(P,b) \leq g(\hat{s})$, where $\hat{s} = \sqrt{1 - c^{1/(2d)}}$ (elementary version). 
\item[iii)] 
For fixed $d$, as $c \to 0$, $S(P,b) \leq g(s_{*}) \to g(1) = 4$;\
\item[iv)] 
$$S(P,b) \leq 1 + 3 \sqrt{ \frac{L_c\ln((2ed)/L_c)}{d}}$$ when $d \geq 15 L_c$
 \item[v)] $$S(P,b) \leq 1 + 3 d^{-1/4} \sqrt{\ln(2e\sqrt{d})} \quad \to 1 \quad \mathrm{as} \quad d \to \infty$$  when $d \geq \max\{L_c^2, 169\}$.
 In particular, if $c\ge e^{-\sqrt{d}}$ and $d\ge 169$, then $S(P,b)\le 2.7169\ldots$
 \end{itemize}
\end{theorem}

\medskip

Recall that to solve Smale's mean value conjecture (Conjecture 1), it suffices to only consider the Schlicht normalized polynomials. Theorem 1 (iv) and (v) show that for Schlicht normalized polynomials whose leading coefficient decays sufficiently slowly (as $d \rightarrow \infty$) the conjecture is asymptotically true. The below Figure \ref{regionOfLeadingCoefficients} illustrates  the values of $c$ for which the upper bound of Theorem 1 (v) is valid: these are the points $(d,c)$ (with $d \geq 169$) contained in the shaded blue region. 

\begin{figure}[H]
    \centering
    \includegraphics[width=\textwidth,height=0.24\textheight,keepaspectratio]{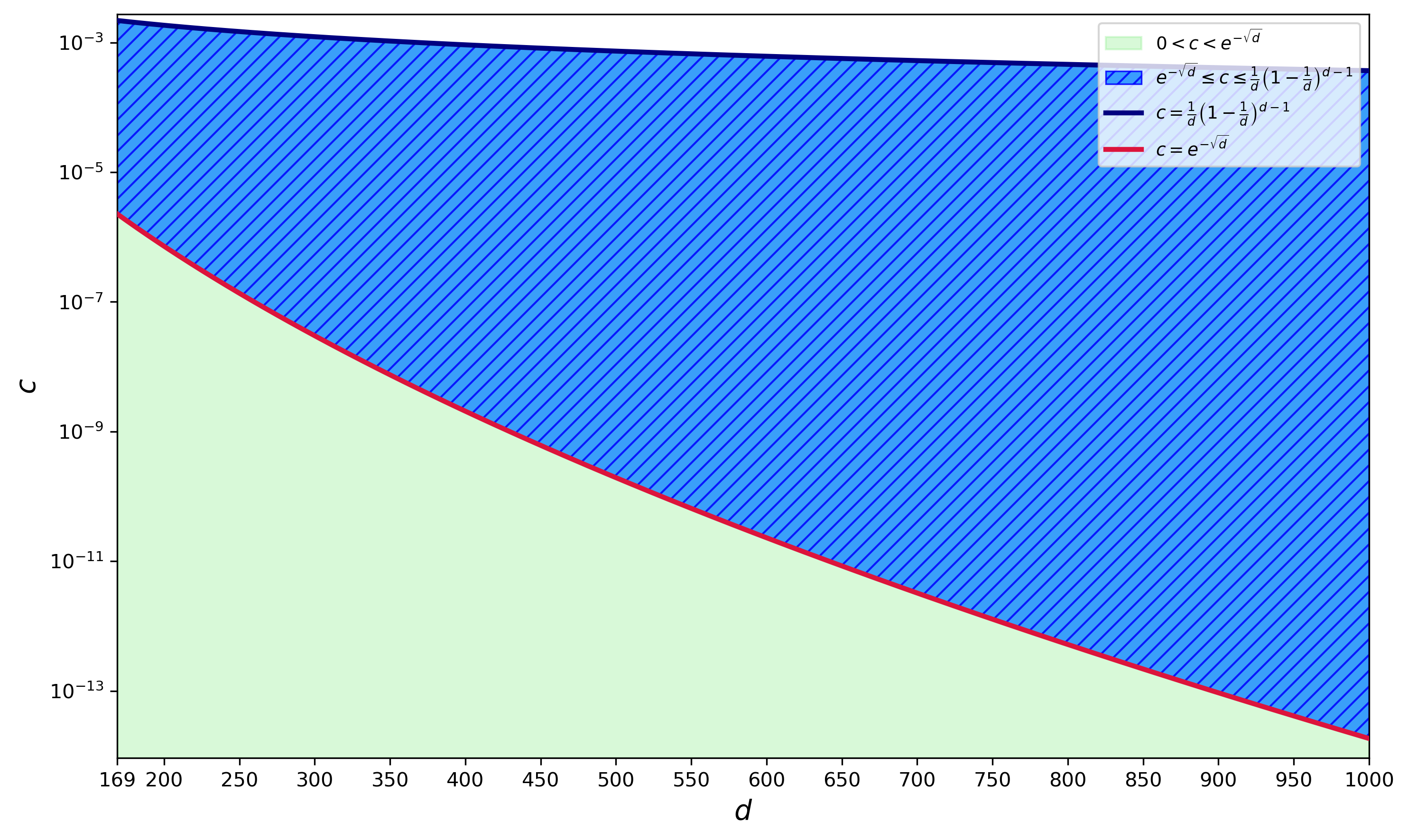}
    \caption{Admissible values of $c$ for Theorem 1 (v)}
    \label{regionOfLeadingCoefficients}
\end{figure}

\indent
To state the next theorem, we first recall the Markov type inequality of Eremenko ($\cite{Eremenko2007}$, Theorem 1) which says 
that if $P$ is a polynomial of degree $d$, and 
$E$ is any continuum $E\subset\C$, then
\begin{equation}\label{Markov}
\mathrm{cap}(E)\,\sup_{z\in E}|P'(z)| \;\le\; 2^{\frac{1}{d}-1}\,d^2\ \sup_{z\in E}|P(z)|.
\end{equation} 
Here $\mathrm{cap}(E)$ is the logarithmic capacity of $E$. If we further assume that $P(0)=0$, $t(P)=1=P'(0)$ and $\Omega = D(P,1)$ is the component of $P^{-1}(\mathbb{D})$ containing $0$, then we have  
\begin{equation}\label{EM}
|P'(0)|\mathrm{cap}(\overline{\Omega}) \le \sup_{z \in \overline{\Omega}} |P'(z)| \mathrm{cap}(\overline{\Omega}) \leq 2^{\frac{1}{d}} \frac{d^2}{2}
\end{equation}

Therefore, if we let $R:=\mathrm{cap}(\overline{\Omega})$ and $A_d:=2^{\frac{1}{d}} \frac{d^2}{2}$, then 
\begin{equation}\label{ieq:R}
   R\le A_d
\end{equation}

On the other hand, applying P\'{o}lya’s result (Proposition \ref{Polya}) to $f=P^{-1}$ restricted to $\mathbb{D}$, we have $1< R=\mathrm{cap}(\overline{\Omega})$. Therefore, $1<R\le A_d$. The following result says that when $R$ is close to the two ends of this range, we have $S(P)\le 3$.\\

\begin{theorem}\label{thm:2} Let $P$ be a polynomial of degree $d\ge 2$ normalized by $P(0)=0$, $P'(0)=1$ and $t(P)=1$. Let
$\Omega=D(P,1)$, $R=\mathrm{cap}(\overline{\Omega})$ and $A_d=2^{\frac{1}{d}}\frac{d^2}{2}$. Then, for any critical point $b \in Z(P') \cap \partial \Omega$,
\begin{equation}\label{Ad-R}
    S(P,b) \le \min\{4-\frac{e^2}{4R},\frac{A_d}{R}, 2+\log\!\Bigl(\frac{A_d}{R}\Bigr)\}.
\end{equation}
Hence, for any $R\in (1,1.8] \cup [\frac{A_d}{3}, A_d]$, we have 
$$
S(P)\le 3.
$$
\end{theorem}

\begin{corollary}\label{cor:1}
Suppose $P$ is a polynomial of degree $d \geq 2$, satisfying  $P(0) = 0$ and $P'(0) \neq 0$ with critical points $b_1,...,b_{d-1}$. For any critical point $b \in Z(P') \cap \partial D(P,t(P))$, the below inequality holds:

$$S(P,b) + \frac{e^2}{4} (\frac{S(P,b)}{M(P,b)})^{\frac{d-1}{d}}  (\frac{2d-1}{d})^{\frac{1}{d}}  < 4 $$
Here $$M(P,b) := (\prod_{i=1}^{d-1} \frac{|b_i|}{|b|})^{\frac{1}{d-1}}$$
In particular, $$|\frac{P(b)}{b P'(0)}| < 3$$ whenever  $$(\prod_{j=1}^{d-1} |\frac{b_j}{b}|)^{\frac{1}{d-1}} < \frac{3 e^2}{4} \approx 5.5$$ and $b \in Z(P') \cap \partial D(P,t(P))$

\end{corollary}
\begin{remark}
For $P$ a degree $d$ Schlicht normalized polynomial with leading coefficient of modulus $c$, in view of the identity (proved in Section \ref{sec2.3})
\begin{equation} \label{leadingCoefficientIdentity} c^{-\frac{1}{d}} = \mathrm{cap}(P^{-1}(\overline{\mathbb{D}})) = d^{\frac{1}{d}}(\prod_{i} |b_i|)^{\frac{1}{d}} = d^{\frac{1}{d}} (\frac{M(P,b)}{S(P,b)})^{\frac{d-1}{d}} \end{equation} Corollary \ref{cor:1} can be rewritten in the form  $$S(P,b) + \frac{e^2}{4} (2d-1)^{\frac{1}{d}}c^{\frac{1}{d}} < 4 $$ or equivalently $$\frac{e^2}{4} (2d-1)^{\frac{1}{d}} \frac{1}{4 - S(P,b)} < c^{-\frac{1}{d}}$$ 

Rewriting this in terms of $c$, we find that $$c < \frac{1}{2d-1}(\frac{4}{e^2} \lbrace 4 - S(P,b)\rbrace)^{d}$$

Thus if $$c \geq \frac{1}{2d-1}(\frac{4}{e^2} A)^{d} $$ then $$S(P,b) \leq 4 - A$$

In particular if $$c \geq \frac{1}{2d-1} (\frac{19}{20})^{d}  $$ then $$S(P,b) \leq 2.245 \ldots$$
\end{remark}

Corollary \ref{cor:1} shows that for a large class of polynomials, for which there exists a ``representative'' critical point on $\partial D(P,t(P))$, i.e a critical point with log-norm close to the mean of the log-norms of the critical points of $P$, the (asymptotically) best known upper bound of $4$ for Conjecture 1 can be replaced by 3. This refines Corollary 2 of \cite{dubinin2012some}.

The third main result is the following improved Eremenko's Markov type inequality which will be used to study the dual mean value conjecture.

\begin{restatable}{theorem}{improvedMarkov} \label{thm:3}
 Let $\mathcal{P}_{d}$ denote the class of polynomials $P$ of degrees between $2$ and $d \geq 2$ inclusive, satisfying $P(0) = 0$, $P'(0) \neq 0$. Then the following inequality holds for any $P \in \mathcal{P}_{d}$.

$$\sup_{z \in \overline{D(P,t(P))}} |P'(z)| \mathrm{cap}(\overline{D(P,t(P))}) \leq (\frac{2}{2d-1})^{\frac{1}{d}} \frac{d^2}{2} t(P)$$ 

The RHS of the above inequality is always strictly smaller than $$\frac{d^2}{2} t(P)$$
\end{restatable}
Theorem \ref{thm:3} gives an improvement on the Markov inequality of Eremenko (\cite{Eremenko2007}, Theorem 1) for the regions $D(P,t(P))$. As a corollary, we obtain a corresponding improvement on the best known lower bound for the dual mean value conjecture.

\begin{restatable}{corollary}{DMVC}\label{cor:2}

Let $d\ge 2$ and $P\in \mathcal{P}_{d}$, $b \in Z(P') \cap \partial D(P,t(P))$. Then: $$S(P,b) \geq S(P_0,b_0) \geq \frac{(d - \frac{1}{2})^{\frac{1}{d}}}{d^2}$$

Here $P_0$ is an extremal polynomial (for $\inf_{P \in \mathcal{P}_{d}} L(P)$) of degree at most $d$ and of degree at least $2$, and $b_0 \in \partial D(P_0,t(P_0)) \cap Z(P_0')$. Moreover $P_0$ may be chosen to satisfy $|P_0(b)| = |P_0(b')|$ for any two points $b,b' \in Z(P_0') \setminus Z(P_0)$, that is, all of $P_0$'s non-zero critical values have the same modulus.

In particular, for any polynomial $P$ of degree $d \geq 2$ satisfying $P(0) = 0$ and $P'(0) \neq 0$ the below inequality holds.

$$D(P) \geq \frac{ (d - \frac{1}{2})^{\frac{1}{d}}}{d^2}$$
\end{restatable}

Corollary \ref{cor:2} improves upon the best known lower bound of $\frac{1}{d^2}$ for Conjecture 2, by a factor $(d - \frac{1}{2})^{\frac{1}{d}}$ for $d \geq 2$. When $d =8$ (the smallest degree for which the dual mean value conjecture is still open), this factor is approximately $1.286$ ; in all cases it is strictly greater than $1$. As $d \rightarrow \infty$, the factor $$(d - \frac{1}{2})^{\frac{1}{d}} \rightarrow 1$$ so we do not obtain an asymptotic improvement on Dubinin's lower bound of $$\frac{1}{d^2}$$ On the other hand, the proof of Corollary \ref{cor:2} actually produces the following sharper result (see also Remark \ref{asymptoticFactorInDualConjecture}).

\begin{restatable}{corollary}{SharperDMVC} \label{SharperDualIneq}
For any polynomial $P \in \mathcal{P}_{d} ; d \geq 2$ and $b \in Z(P') \cap \partial D(P,t(P))$, letting $$\gamma$$ denote the Jordan arc from $0$ to $b$ in $\overline{ D(P,t(P))}$ for which $P(\gamma) = [0, P(b)]$ the below inequality holds
$$
S(P,b) \geq \frac{\mathrm{cap}(\gamma)}{\mathrm{cap}[0,b]}\frac{\left(d-\frac{1}{2}\right)^{\frac{1}{d}}}{d^2}
\geq \frac{\left(d - \frac{1}{2}\right)^{\frac{1}{d}}}{d^2}
$$

\end{restatable}

Theorems \ref{thm:1},\ref{thm:2} and Corollary \ref{cor:1} (our main results for the mean value conjecture) are established by means of a strengthened version of the Koebe $\frac{1}{4}$-Theorem for Schlicht functions mapping the unit disk onto sets having bounded logarithmic capacity, due to Cunningham \cite{Cunningham1993}, whereas Corollaries \ref{cor:2} and \ref{SharperDualIneq} (our results for the dual mean value conjecture) are derived from Theorem \ref{thm:3}. 

This paper uses the notion of logarithmic capacity. See e.g. \cite{ransford1995potential},\cite{saff-totik2024} for the definition and basic properties of logarithmic capacity.\\

\subsection{Notation and Background} \label{Notation}
We record here basic notation to be used throughout this paper.
Whenever $Q$ is a polynomial, we will let $$Z(Q)$$ refer to its multi-set of zeros, that is, $Z(Q)$ contains the zeros of $Q$ recorded with multiplicity. Moreover, $$\mathbb{D}_{r}$$ will refer to the $\textbf{open}$ disk with center zero and radius $r > 0$, while $$\overline{\mathbb{D}_{r}}$$ will refer to its closure. More generally we will denote the open disk of radius $r > 0$ and center $z_0 \in \mathbb{C}$ by $$\Delta(z_0;r) \coloneq \{z \in \mathbb{C} : |z - z_0| < r \}$$ and set $$\Delta(R) \coloneq \{z \in \overline{\mathbb{C}} : |z| > R \}$$  $$\Delta \coloneq \Delta(1)$$

We recall here the definition of $t(P)$ (see also the discussion at the beginning of Section \ref{preliminaryConstructionsandStatementOfResults}). We define, for $P$ a non-linear polynomial fixing the origin with non-zero derivative at the origin $t(P)$ to be the maximum of the quantities $|P(b)|$, where the maximum is taken over the set of critical points $b$ of $P$ for which $P$ is a univalent conformal mapping of the component $D(P,|P(b)|)$ of $P^{-1}(\mathbb{D}_{|P(b)|})$ containing $0$ onto $\mathbb{D}_{|P(b)|}$. Then $t(P) = |P(b_0)|$ for some critical point $b_0$ of $P$ lying on $\partial D(P,t(P))$, and $D(P,t(P))$ is a Jordan domain bounded by a piecewise-analytic Jordan curve (see the discussion at the start of Section \ref{preliminaryConstructionsandStatementOfResults}), for which $P$ effects a univalent mapping of $D(P,t(P))$ onto $\mathbb{D}_{t(P)}$. Indeed, $D(P,t(P))$ is the largest such domain (containing the origin), which we refer to as the \textbf{maximal univalent lemniscate} of $P$ (centered at the origin). We let $$E(H,r)$$ and $$D(H,r)$$ (respectively) denote the components of $$H^{-1}(\overline{\mathbb{D}_{r}})$$ and $$H^{-1}(\mathbb{D}_{r})$$ containing the origin, whenever $H$ is a polynomial fixing the origin.

Define  $$U(P) := \max_{b \in Z(P') \cap \partial D(P,t(P)) } |\frac{P(b)}{b P'(0)}|$$ $$L(P) \coloneq \min_{b \in Z(P') \cap \partial D(P,t(P))} |\frac{P(b)}{b P'(0)}|$$ and $$M(P,b_0) := (\prod_{b \in Z(P')} \frac{|b|}{|b_0|})^{\frac{1}{d-1}}$$ for a distinguished critical point $b_0 
\in Z(P')$. Finally we let $$M(P) := \max_{b \in Z(P')} M(P,b)$$ $$S(P,b) := |\frac{P(b)}{b P'(0)}|$$ for a critical point $b \in Z(P')$ and $$S(P) \coloneq \min_{b \in Z(P')} S(P,b)$$ $$D(P) \coloneq \max_{b \in Z(P')} S(P,b)$$

Note that  $U(P)$,$L(P)$, $S(P)$,$D(P)$ and $S(P,b)$, $M(P,b)$ are invariant under rescalings $P \rightarrow \alpha P (\beta z)$ of the polynomial $P$ and that we have the following two inequalities:

$$S(P) \leq U(P)$$
$$L(P) \leq D(P)$$

It follows that obtaining upper and lower bounds for $U(P)$ and $L(P)$ will in turn establish upper and lower bounds for $S(P)$ and $D(P)$ respectively. 

The choice of restriction to the set $Z(P') \cap \partial D(P,t(P))$ for $U(P)$ was first considered by Dubinin in the paper \cite{dubinin2012some}, where the below result was proved.

\begin{result}[Dubinin; \cite{dubinin2012some} Corollary 3]
     Let $P(z) = a_dz^d + ... + a_1z$ , $a_d, a_1 \neq 0$, $d \geq 2$, then $$S(P,b_0) +  \frac{1}{M(P,b_0) (d-1)^{\frac{1}{d}}} \leq 4 $$

    holds for any critical point $b_0 \in \partial D(P,t(P))$.
\end{result} 

  For the quantity $L(P)$, Dubinin has mentioned in \cite{dubinin2012some} that it would be interesting to consider its upper bound in connection to his Problem 3.  On the other hand, Dubinin's finite increment theorem \cite{Dubinin2010} actually implies the following lower bound for $L(P)$:

\begin{result}[Dubinin; \cite{Dubinin2010} Theorem 2]
    Let $P(z) = a_dz^d + ... + a_1z$, $a_d,a_1 \neq 0$, $d \geq 2$, then $$L(P) \geq \frac{1}{d^2}$$ 
\end{result}
The main results of this paper provide, as corollaries, refinements of the above results of Dubinin. In particular Corollary \ref{cor:1} refines Result 1, and Corollaries \ref{cor:2},\ref{SharperDualIneq} refine Result 2.

\medskip
In Section \ref{sec:2} we develop the tools needed for the two mean value conjectures.  We first study an extremal problem for the logarithmic capacity of polynomial lemniscates and establish a sharp comparison principle (Theorem \ref{thm:5}), with the polynomial $P^\ast(z)=z^d-z$ playing an extremal role.  We then deduce Corollary \ref{cor:1} from Theorem \ref{thm:2} (Section \ref{sec2.3}), and give the proofs of the quantitative bounds for the mean value conjecture (Theorems \ref{thm:1} and \ref{thm:2}) in Section \ref{sec2.4}, combining capacity estimates with Cunningham’s refined Koebe $\frac14$--theorem. Section \ref{sec3} concerns the extremal problem for the boundary quantity $U(P)$ and proves the existence and structure of extremal polynomials for $\sup_{P \in \mathcal{P}_{d}} U(P)$ (Theorem \ref{thm:9}), using a quasiconformal deformation argument in the spirit of Eremenko--Hayman. In Section \ref{sec:4} we prove the improved Markov-type inequality for the maximal univalent lemniscates $\overline{D(P,t(P))}$ (Theorem \ref{thm:3}), refining Eremenko’s inequality by incorporating the capacity comparison principle from Section \ref{sec:2} (Theorem \ref{thm:5}).  Corollary \ref{cor:2} (the improved unconditional lower bound for the dual mean value conjecture) is then proved at the end of Section \ref{sec4.1} and Corollary \ref{SharperDualIneq} is obtained in the final Section \ref{sec4.2}.

\medskip  

\section{Proofs for Mean Value Conjecture results} \label{sec:2}
In this section, we will prove our results for the mean value conjecture (Theorems \ref{thm:1},\ref{thm:2} and Corollary \ref{cor:1}). In Section \ref{sec2.3} we prove Corollary \ref{cor:1} from Theorem \ref{thm:2} while Theorems \ref{thm:1} and \ref{thm:2} are established in Section \ref{sec2.4}. Section \ref{sec2.1} establishes an extremal property of the polynomial $P^{\ast}(z) = z^d - z$ which will turn out to be key to the proof of Theorem \ref{thm:3}, and is (less essentially) also used in the proof of Corollary \ref{cor:1}.

We state here for future reference a result (Theorem 5.2.5 of \cite{ransford1995potential}) which will be frequently used in our paper.\\

\begin{proposition}\label{prop4}
Let $E \subset \mathbb{C}$ be a compact set, and $P \in \mathbb{C}[z]$ a degree $d$ polynomial. Then

$$\mathrm{cap}(P^{-1}(E)) = (\frac{\mathrm{cap}(E)}{|a_d|})^{\frac{1}{d}}$$
Here $\mathrm{cap}$ denotes the logarithmic capacity.
\end{proposition}
\subsection{An extremal problem for logarithmic capacity} \label{sec2.1}
In this section, we will establish the following result.

\begin{theorem}\label{thm:5}
Let $P$ be a degree $d \geq 2$ polynomial, satisfying $P'(0) \neq 0$ and $P(0) = 0$. Write $D(P,t(P))$ for the component of $P^{-1}(\mathbb{D}_{t(P)})$ containing the origin. Then the following inequality holds:

$$\frac{\mathrm{cap}(\overline{D(P,t(P))})}{\mathrm{cap}(P^{-1}(\overline{\mathbb{D}_{t(P)}}))} \leq  \frac{\mathrm{cap}(\overline{D(P^{\ast},t(P^{\ast}))})}{\mathrm{cap}({(P^{\ast})^{-1}(\overline{\mathbb{D}_{t(P^{\ast})}}))}}  \eqqcolon C(d)^{-1} = (2d-1)^{-\frac{1}{d}}$$

Equality holds in the first inequality for linear rescalings of $P^{\ast}(z) = z^d - z$
\end{theorem}

The below figures show how the \textbf{central components} $\overline{D(P^{\ast},t(P^{\ast}))}$ i.e. the closures of the regions bounded by the single Jordan curve encircling the origin, compare in size to the entire lemniscate $(P^{\ast})^{-1}(\overline{\mathbb{D}_{t(P^{\ast})}})$ for the extremal polynomial $P^{\ast}$ of Theorem \ref{thm:5}. As $d \rightarrow \infty$, one sees that the central components become asymptotically circular.

\begin{figure}[ht]
   \begin{subfigure}{0.2\textwidth}
       \includegraphics[width = \linewidth]{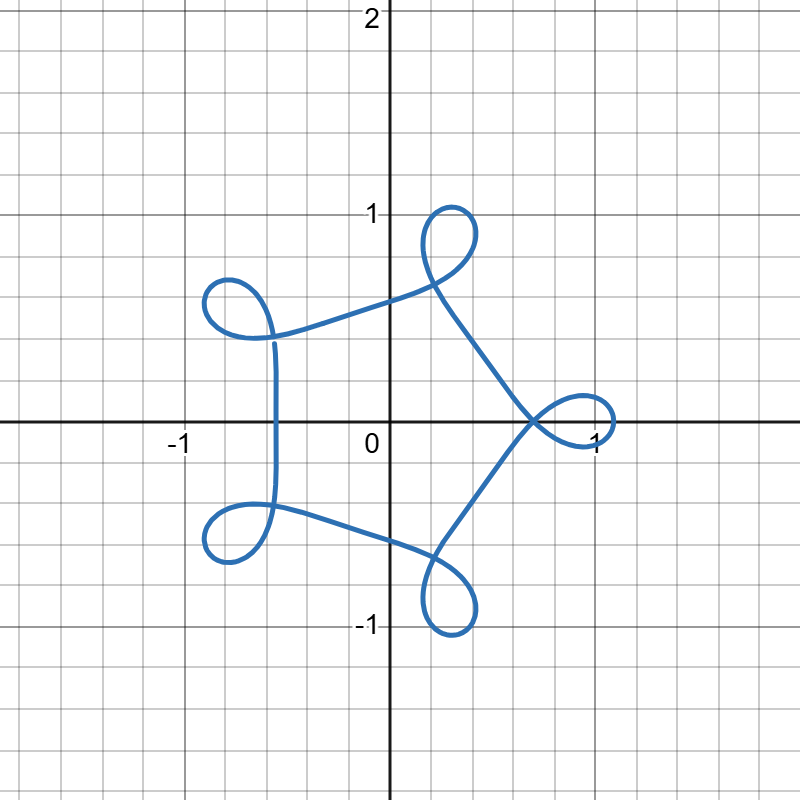} 
       \caption{Critical lemniscate of $z^6 - z$}
       \label{fig:subim1}
   \end{subfigure}
\hfill 
     \begin{subfigure}{0.2\textwidth}
       \includegraphics[width = \linewidth]{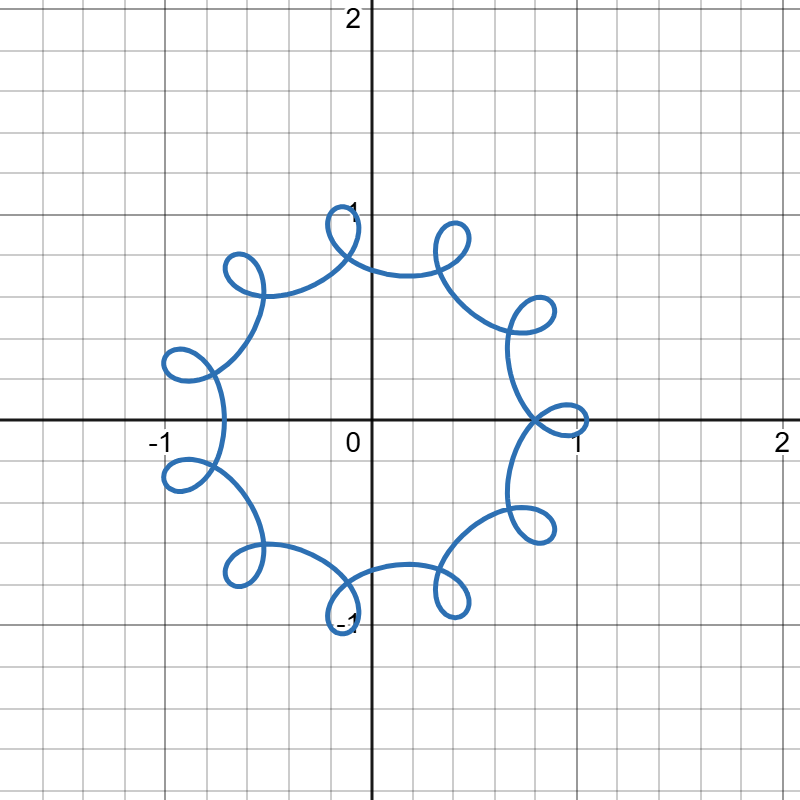} 
       \caption{Critical lemniscate of $z^{12} - z$}
       \label{fig:subim2}
   \end{subfigure}
\hfill 
    \begin{subfigure}{0.2\textwidth}
       \includegraphics[width = \linewidth]{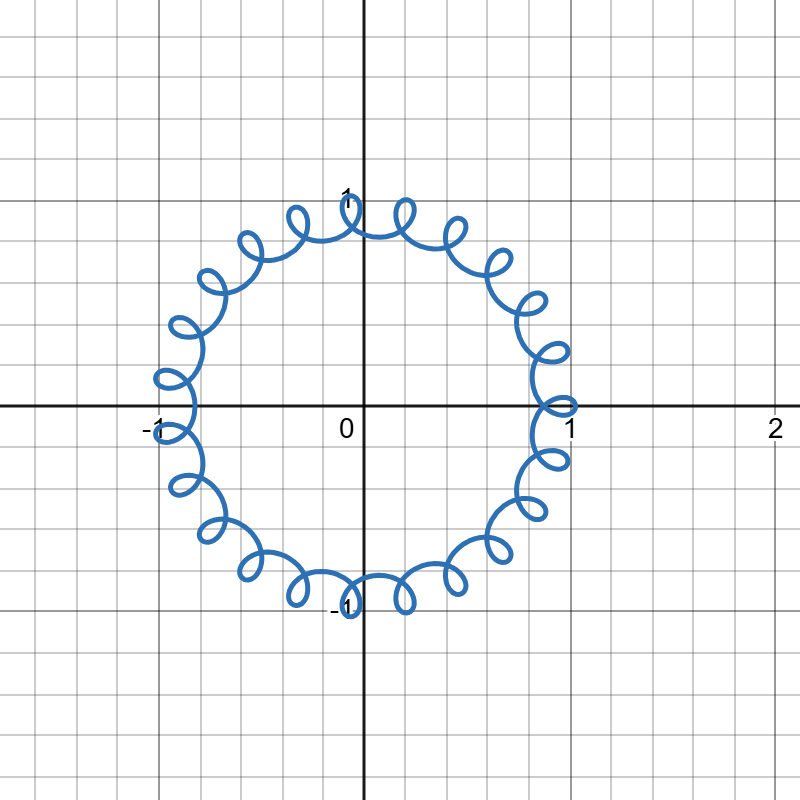} 
       \caption{Critical lemniscate of $z^{24} - z$}
       \label{fig:subim3}
   \end{subfigure}
\hfill 
    \begin{subfigure}{0.2\textwidth}
       \includegraphics[width = \linewidth]{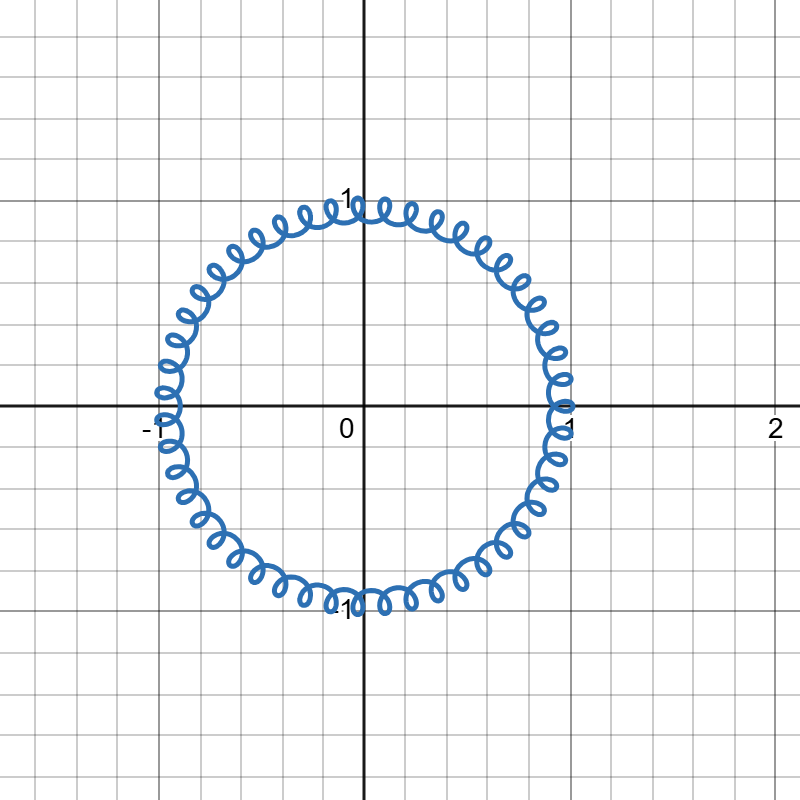} 
       \caption{Critical lemniscate of $z^{48} - z$}
       \label{fig:subim4}
   \end{subfigure}

\caption{Critical lemniscates of $z^d-z$ for $d=6,12,24,48$}
\label{fig:image2}
   
\end{figure}

\begin{remark}
Note that the ratio $$\frac{\mathrm{cap}(\overline{D(P,t(P))})}{\mathrm{cap}(P^{-1}(\overline{\mathbb{D}_{t(P)}}))}$$ is invariant under rescalings of $P$, i.e. if $Q(z) = \alpha P(\beta z)$ for some non-zero complex numbers $\alpha, \beta \in \mathbb{C}$, then $$\frac{\mathrm{cap}(\overline{D(P,t(P))})}{\mathrm{cap}(P^{-1}(\overline{\mathbb{D}_{t(P)}}))} = \frac{\mathrm{cap}(\overline{D(Q,t(Q))})}{\mathrm{cap}(Q^{-1}(\overline{\mathbb{D}_{t(Q)}}))}$$ Thus in light of the fact that the Schlicht normalization of $$P^{\ast}(z) = z^d - z$$ has the form $$Q^{\ast} (z) \coloneq \frac{1}{t(P^{\ast})} P^{\ast}( -t(P^{\ast}) z)$$ and $$t(P^{\ast}) = \frac{d-1}{d^{\frac{d}{d-1}}}$$ we observe that the central components $$D(P^{\ast},t(P^{\ast})) = t(P^{\ast}) D(Q^{\ast},1) = \frac{d-1}{d^{\frac{d}{d-1}}} D(Q^{\ast},1)$$ are asymptotically given by the unit disk $\mathbb{D}$, by (the equality case of) P\'olya's Theorem (\ref{Polya}) and the fact that $$1 \leq \mathrm{cap}(\overline{D(Q^{\ast},1)}) = (2d-1)^{-\frac{1}{d}} t(P^{\ast})^{-\frac{d-1}{d}} \rightarrow 1$$ In fact the same argument shows that for any sequence of Schlicht normalized polynomials $P_{d_k} ; k=1,2,...$ of increasing degrees $\deg(P_{d_k}) = d_k \uparrow \infty$, for which $(d,c_{d_k}) ; k=1,2, \ldots$ lies in the blue region of Figure 1 (or more generally such that $\lim_{k \rightarrow \infty} c_{d_k}^{-\frac{1}{d_k}} = 1$), one has that $D(P_{d_k},1)$ is asymptotically given by $\mathbb{D}$ as $k \rightarrow \infty$. 
\end{remark}
\subsection{Proof of Theorem \ref{thm:5}} \label{sec2.2}
\begin{proof}
Let $P$ be a degree $d $ polynomial, satisfying $P'(0) \neq 0$ and $P(0) = 0$ (as in the statement of Theorem \ref{thm:5}). Since the left-hand side of the inequality of Theorem \ref{thm:5} is invariant under rescalings of the form $P \rightarrow A P(B z)$, we may assume that $t(P) = 1$. Let $$f : \Delta \rightarrow \overline{\mathbb{C}} \setminus \overline{D(P,1)}$$ be the exterior conformal map fixing $\infty$ and satisfying $$f(z) \sim \mathrm{cap}(\overline{D(P,1)})z$$ near $\infty$. Let $$\alpha \coloneq \mathrm{cap}(\overline{D(P,1)}) $$

Define $$g(w) \coloneq P(f(w))$$ Since $\partial D(P,1)$ is locally-connected (in fact, it is a piecewise analytic Jordan curve by the argument given at the start of Section \ref{preliminaryConstructionsandStatementOfResults}), by Carath\'eodory's theorem (see Theorem 2.6 of page 24 of \cite{pommerenke1992boundary}) we conclude that $f$ extends continuously to $\overline{\Delta}$, i.e. to the boundary $\partial \Delta = \partial \mathbb{D}$, and $f(w) \in \partial D(P,1)$ for $w \in \partial \mathbb{D}$, so that $|g(w)| = 1$ for $w \in \partial \mathbb{D}$. By the Schwarz reflection principle, $g$ is a rational function. Since there are exactly $d-1$ zeros of $P$ (counting multiplicity) lying exterior to $D(P,1)$, $g$ has exactly $d-1$ zeros (counting multiplicity) lying in $\Delta$ (which are all finite), and since $g$ is defined in $\mathbb{D}$ by reflection in the unit circle, $g$ has exactly $d-1$ poles (counting multiplicity) lying in $\mathbb{D}$, which are all contained in $\mathbb{D} \setminus \{0 \}$. Since $f(z) \sim \alpha z$ near $\infty$, $g$ has a pole of order $d$ at $\infty$, and no other poles in $\Delta$. By reflection, $g$ has a zero of order $d$ at the origin, and has no other zeros in $\mathbb{D}$. Collectively, these points comprise all of the zeros and poles of $g$. It follows that $$g(w) = \lambda \frac{w^d}{B(w)}$$ where $|\lambda| = 1$ is a unimodular constant and $B$ is a finite Blaschke product of degree $d-1$ of the form $$B(w) = \prod_{j=1}^{d-1} \frac{w - A_{j}}{\overline{A_{j}}w - 1}$$ where $$A_{j} \in \mathbb{D} \setminus \{0 \} ; j=1,...,d-1$$  we find that for $w$ near $\infty$, $$g(w) \sim \lambda \lbrace \prod_{j=1}^{d-1} \overline{A_{j}} \rbrace  \ w^d$$ On the other hand, since $g(w) = P(f(w))$ we have that $$g(w) \sim a_d \alpha^d w^d$$ near $\infty$. Since $$|\lambda \prod_{j=1}^{d-1} \overline{A_{j}}| = |B(0)|$$ we obtain that $$|B(0)| = |a_d| \alpha ^d$$

By Proposition \ref{prop4} $$|a_d| \alpha^d = (\frac{\mathrm{cap}(\overline{D(P,1)})}{\mathrm{cap}(P^{-1}(\overline{\mathbb{D}}))})^{d} = |B(0)|$$
It follows that the problem of obtaining an upper bound for the capacity ratio is equivalent to obtaining an upper bound for $|B(0)|^{\frac{1}{d}}$, which we will now proceed to obtain. Since $f$ is an orientation preserving homeomorphism of the unit circle $\partial \mathbb{D}$ onto $$\partial D(P,1)$$ (here and always we assume $$\partial D(P,1)$$ is oriented so that the origin lies to the left, and similarly for $\partial \mathbb{D}$) and $P$ is an orientation preserving homeomorphism of $$\partial D(P,1)$$ onto $\partial \mathbb{D}$ we conclude that $$g= P \circ f$$ is an orientation preserving homeomorphism of $$\partial \mathbb{D}$$ onto itself, so that we necessarily have $$\frac{d}{d \theta} \arg{g(\exp(i \theta))} =\frac{d}{d \theta} \arg{(P\circ f )(\exp(i \theta))} = d - \frac{d}{d \theta} \arg{B(\exp(i \theta))} \geq 0$$
As $B$ is a finite Blaschke product, $$\frac{d}{d \theta} \arg{B(\exp(i \theta))} = |B'(\exp(i \theta))|$$ and we conclude $$d \geq |B'(\exp(i \theta))|$$ for every $\theta \in [0, 2 \pi]$, hence \begin{equation} \label{BlaschkeUpper} d \geq ||B'||_{ \partial \mathbb{D}} \end{equation}

For finite Blaschke products $\phi$ of degree $k$, and an arbitrary point $\lambda \in \partial \mathbb{D}$, the following representation is valid (see Theorem 8.1 of \cite{cowen1982inequalities}) \begin{equation} \label{CowenPomm} \sum_{j=1}^{k} \frac{1}{|\phi'(\zeta_{j})|} = \mathrm{Re}(\frac{\lambda + \phi(0)}{\lambda - \phi(0)}) \end{equation} where $\zeta_{j} ; j=1,...,k$ are the $k$ distinct points on $\partial \mathbb{D}$ for which $\phi(\zeta_{j}) = \lambda$.

Suppose $\phi(0) \neq 0$. Since the right-hand side of (\ref{CowenPomm}) is the Poisson kernel $$\frac{1 - |\phi(0)|^2}{|\lambda - \phi(0)|^2}$$ we conclude that the right-hand side of (\ref{CowenPomm}) is equal to  $$\frac{1 - |\phi(0)|}{1 + |\phi(0)|}$$ for $$\lambda = -\frac{\phi(0)}{|\phi(0)|}$$ 

On the other hand, the left-hand side of (\ref{CowenPomm}) is always greater than or equal to $$\frac{k}{||\phi'||_{\partial \mathbb{D}}}$$ We thus obtain the following inequality.

\begin{equation} \label{BlaschkeIneq} \frac{k}{||\phi'||_{\partial \mathbb{D}}} \leq \frac{1 - |\phi(0)|}{1 + |\phi(0)|} \end{equation}
for any degree $k$ Blaschke product $\phi$ satisfying $\phi(0) \neq 0$. Applying (\ref{BlaschkeIneq}) to $\phi = B$ we obtain \begin{equation} \label{BlaschkeLower} ||B'||_{\partial \mathbb{D}} \geq (d-1) \frac{1 + |B(0)|}{1 - |B(0)|} \end{equation}

Combining (\ref{BlaschkeUpper}) and (\ref{BlaschkeLower}) we find \begin{equation} \label{BlaschkeCombined} \frac{d}{d-1} \geq \frac{1 + |B(0)|}{1 - |B(0)|} \end{equation}

This is equivalent to \begin{equation} \label{BlaschkeFinalWithoutRoot} |B(0)| \leq (2d-1)^{-1} \end{equation} So that we obtain \begin{equation} \label{BlaschkeFinalWithRoot} |B(0)|^{\frac{1}{d}} \leq (2d-1)^{-\frac{1}{d}} \end{equation}

It remains only to show \begin{equation} \label{equalityCap} C(d)^{-1} \coloneq \frac{\mathrm{cap}(\overline{D(P^{\ast},t(P^{\ast}))})}{\mathrm{cap}({(P^{\ast})^{-1}(\overline{\mathbb{D}_{t(P^{\ast})}}))}}  = (2d-1)^{-\frac{1}{d}} \end{equation} We now proceed to prove this.

Let $$P^{\ast}(z) = z^d - z$$ and let $$t \coloneq t(P^{\ast}) = \frac{d-1}{d^{\frac{d}{d-1}}}$$  $$Q(u) \coloneq u(u-1)^{d-1}$$ It is easy to see that the pre-image of $\overline{D(Q,t^{d-1})}$ by the power map $$z \rightarrow u=z^{d-1}$$ is exactly $$\overline{D(P^{\ast},t)}$$ Hence, by Proposition \ref{prop4} we have $$\mathrm{cap}(\overline{D(P^{\ast},t)}) = \mathrm{cap}(\overline{D(Q,t^{d-1})})^{\frac{1}{d-1}}$$

We now compute $\mathrm{cap}(\overline{D(Q,t^{d-1})})$. Let $f$ be the conformal map of the exterior of the unit disk $\Delta$ onto the exterior of $D(Q,t^{d-1})$ fixing $\infty$ and satisfying $f(w) \sim c w$ near $w = \infty$, where $$c \coloneq \mathrm{cap}(\overline{D(Q,t^{d-1})})$$ For $w \in \partial \mathbb{D}$, $f(w) \in \partial D(Q,t^{d-1})$ so that $$|f(w)(f(w) - 1)^{d-1}| = t^{d-1}$$ Then by the Schwarz reflection principle, $g(w) \coloneq Q(f(w)) = f(w) (f(w) - 1)^{d-1}$ extends to a rational function $g : \overline{\mathbb{C}} \rightarrow \overline{\mathbb{C}}$. The zeros of $g$ in the exterior of the unit disk are precisely the pre-images (by $f$) of the zeros of $Q$ lying exterior to $D(Q,t^{d-1})$ by definition of $g$, and since $g$ is defined in the unit disk by Schwarz reflection, the zeros of $g$ lying in the unit disk are exactly the reflections in the unit circle of the poles of $g$ in the exterior of the unit disk. Since $g(w) = f(w)(f(w) - 1)^{d-1}$ and $f(w) \sim cw$ near $\infty$, $g$ has a pole of multiplicity $d$ at $\infty$. Since the only zero of $Q$ in the exterior of $D(Q,t^{d-1})$ is $1$, we conclude that $g$ has a zero at the unique point $W_0 \in \overline{\mathbb{C}} \setminus \overline{\mathbb{D}}$ for which $$f(W_0) = 1$$ Moreover, near $W_0$, we have $g(w) = f(w)(f(w) - 1)^{d-1} \sim A (w - W_0)^{d-1}$ where $A = f'(W_0)^{d-1}$, so that the multiplicity of $W_0$ as a zero of $g$ is exactly $d-1$. By symmetry, $g$ has a zero of multiplicity $d$ at $0$ and a pole of multiplicity $d-1$ at the point $\frac{1}{\overline{W_0}}$. Collectively, these four points comprise all the zeros and poles of $g$. It follows that $$g(w) = t^{d-1}\lambda w^d (\frac{w - W_0}{\overline{W_0}w - 1})^{d-1}$$ where $|\lambda| = 1$ is a unimodular constant. Note that the above construction (which is the same as the construction used in the first part of the proof) is essentially the argument used in the proof of the first part of Theorem 2.1 in \cite{Younsi2016} on page 4.

Now letting $f(w) \sim cw$ as $w \rightarrow \infty$, we have $g(w) \sim c^d w^d$ as $w \rightarrow \infty$, and $$\mathrm{cap}(\overline{D(Q,t^{d-1})})^{d} =c^d = t^{d-1}(\frac{1}{|W_0|})^{d-1}$$ so that $$\mathrm{cap}(\overline{D(Q,t^{d-1})}) = t^{\frac{d-1}{d}}|W_0|^{-\frac{d-1}{d}}$$ 

It remains to find $W_0$. Notice that $\frac{1}{d}$ is the unique critical point of $Q$ lying on $\partial D(Q,t^{d-1})$ (and is the only critical point of $Q$ besides $w = 1$). By the Schwarz reflection principle and conjugate symmetry of $\overline{D(Q,t^{d-1})}$, $f(z) = \overline{f(\overline{z})}$ necessarily maps $[1, +\infty)$ onto $[\frac{1}{d} ,+\infty)$ due to the normalization $f(w) \sim cw$ where $c > 0$ at $\infty$. In particular $W_0 \in (1,+\infty)$ is real. Now $$g'(w) = f'(w)(f(w)-1)^{d-1} + (d-1)f(w)f'(w) (f(w) -1)^{d-2} $$ $$= (f(w)-1)^{d-2} (f'(w) (f(w)-1) + (d-1)f(w) f'(w))  $$ $$= (f(w)-1)^{d-2} ( f'(w)) (d f(w) -1)) $$ and \begin{equation} \label{OutwardCorner} \lim_{w \in \Delta ; w \rightarrow 1 } |f'(w)| = 0 \end{equation} Where in the equality (\ref{OutwardCorner}) the limit is equal to $0$ because $f(1) = \frac{1}{d}$ is a corner of the curve $\partial D(Q,t^{d-1})$ of angle $\frac{3 \pi}{2}$ (of the exterior domain). We thus conclude that $$g'(1)=0$$

Now $$\frac{g'(w)}{g(w)} = \frac{d}{w} + \frac{d-1}{w - W_0} - \frac{(d-1)(W_0)}{W_0w -1}$$

Evaluating this at $w=1$ we find $$ d + (d-1)\lbrace \frac{1}{1-W_0} - \frac{W_0}{W_0-1} \rbrace = d + (d-1)\lbrace \frac{1+ W_0}{1-W_0} \rbrace = 0$$ Hence $$W_0 = (2d-1)$$

We conclude that $c = t^{\frac{d-1}{d}}(2d-1)^{-\frac{d-1}{d}}$ and $$C(d)^{-1} \coloneq \frac{\mathrm{cap}(\overline{D(P^{\ast},t(P^{\ast}))})}{\mathrm{cap}({(P^{\ast})^{-1}(\overline{\mathbb{D}_{t(P^{\ast})}}))}}  = \frac{c^{\frac{1}{d-1}}}{t^{\frac{1}{d}}} = \frac{t^{\frac{1}{d}}(2d-1)^{-\frac{1}{d}}}{t^\frac{1}{d}}  = (2d-1)^{-\frac{1}{d}}$$

This completes the proof of Theorem \ref{thm:5}.
\end{proof}

\begin{remark}
Theorem \ref{thm:5} (in fact, a more general comparison principle) can also be established by an argument involving condenser capacity and the so-called method of dis-symmetrization (due to Dubinin). Since this approach would require familiarity with these notions, for the purposes of keeping the proofs of this paper relatively self-contained, we omit this argument.
\end{remark}

\subsection{Proof of Corollary \ref{cor:1} from Theorem \ref{thm:2}} \label{sec2.3}
Let $P$ be Schlicht normalized , so that $t(P) = 1 = |P'(0)|$. Then $$\frac{1}{S(P,b)} = |b|$$ holds for any critical point $b \in \partial D(P,1) \cap Z(P') \equiv \partial D(P,t(P)) \cap Z(P')$ (in fact, this holds for any critical point $b \in Z(P')$ satisfying $|P(b)| = 1 = t(P)$). By Proposition \ref{prop4} we have $$\mathrm{cap}(P^{-1}(\overline{\mathbb{D}})) = (\frac{1}{|a_d|})^{\frac{1}{d}}$$ Since $$d |a_d| \prod_{i=1}^{d-1} |b_i| = |P'(0)| = 1$$ we obtain $$\mathrm{cap}(P^{-1}(\overline{\mathbb{D}})) = d^{\frac{1}{d}}(\prod_{i=1}^{d-1} |b_i|)^{\frac{1}{d}} = d^{\frac{1}{d}} (\frac{M(P,b)}{S(P,b)})^{\frac{d-1}{d}}$$ where the final equality follows from the definition of $M(P,b)$ (see Section \ref{Notation}) as $$M(P,b) = (\prod_{i=1}^{d-1} \frac{|b_i|}{|b|})^{\frac{1}{d-1}}$$ In the above computations, $$b_i ; i=1,...,d-1$$ denote the critical points of $P$ counted with multiplicity.

We thus obtain from Theorem \ref{thm:5} that \begin{equation} \label{eqn:3} \mathrm{cap} (\overline{D(P,1)}) \leq C(d)^{-1} d^{\frac{1}{d}} (\frac{M(P,b)}{S(P,b)})^{\frac{d-1}{d}} \end{equation} for any critical point $b \in \partial D(P,1) \cap Z(P')$. 

Recall that $C(d)^{-1} = (2d-1)^{-\frac{1}{d}}$ in Theorem \ref{thm:5}, we then have:

\begin{equation} \label{eqn:4} (\frac{S(P,b)}{M(P,b)})^{\frac{d-1}{d}} \mathrm{cap}(\overline{D(P,1)}) \leq C(d)^{-1} d^{\frac{1}{d}} = \frac{d^{\frac{1}{d}}}{(2d-1)^{\frac{1}{d}}}  \end{equation}
for any critical point $b \in Z(P') \cap \partial D(P,1) \equiv Z(P') \cap \partial D(P,t(P))$.

Theorem \ref{thm:2} and (\ref{eqn:4}) immediately imply Corollary \ref{cor:1}, which in turn implies the following

\begin{theorem}
 Let $P(z) = a_dz^d + ... + a_1z$ , $a_d, a_1 \neq 0$, $d \geq 2$, then $$U(P) + \frac{e^2}{4} (\frac{U(P)}{M(P)})^{\frac{d-1}{d}}  (\frac{2d-1}{d})^{\frac{1}{d}}  < 4 $$
\end{theorem}

\medskip  

\subsection{Proofs for Theorems \ref{thm:1} and \ref{thm:2}} \label{sec2.4}

The proof of Theorems \ref{thm:1} and \ref{thm:2} will require the below (Proposition \ref{Polya}) classical result of P\'olya (see e.g. \cite{ransford1995potential} page 141, Theorem 5.3.5), as well as a strengthened version of the Koebe $\frac{1}{4}$-Theorem for Schlicht functions mapping the unit disk onto sets of bounded logarithmic capacity by Cunningham (Theorem \ref{Cunningham}).

\begin{proposition}\label{Polya}
Let $f : \mathbb{D} \rightarrow \mathbb{C}$ be a Schlicht univalent function. That is, $f$ is univalent and satisfies $f'(0) = 1$. Then,
$$\mathrm{cap}(\overline{f(\mathbb{D})}) \geq 1$$
Equality in the above inequality holds if and only if $f$ is the identity function.
\end{proposition}

The following refined Koebe $\frac14$--type theorem of Cunningham for Schlicht functions whose images have controlled logarithmic capacity plays the crucial role in the proof of Theorems \ref{thm:1} and  \ref{thm:2}.

\begin{theorem}[Cunningham 1993 \cite{Cunningham1993}]\label{Cunningham}
For every $R > 1$, let $\mathcal{S}_{R}$ denote the class of Schlicht functions (i.e. univalent functions $f: \mathbb{D} \rightarrow \mathbb{C}$ fixing the origin and satisfying $f'(0) = 1$), whose images have logarithmic capacity $R$.
Then every function $f \in \mathcal{S}_{R}$ maps $\mathbb{D}$ onto a region containing the open disk about the origin with radius 
\begin{equation}\label{rA}
   r = \frac{A-1}{4A}(\frac{\sqrt{A} + 1}{\sqrt{A}-1})^{\frac{1}{\sqrt{A}}} 
\end{equation}
where $A$ is the unique positive value satisfying the relation 
\begin{equation}\label{RA}
R = \frac{(A-1)^2}{16 A} ( \frac{ \sqrt{A} + 1}{ \sqrt{A} - 1})^{\frac{1}{\sqrt{A}} + \sqrt{A}} 
\end{equation}
\end{theorem}
\medskip
\noindent
{\it Proof of Theorem \ref{thm:1}.} We will actually prove that Theorem \ref{thm:1} still holds when $S(P)$ is replaced by $U(P)$. This is a stronger result as $S(P)\le U(P)$.\\

Since $P$ has the Schlicht normalization. The map $$f:= P^{-1}:\mathbb{D}\to \Omega:=D(P,1)$$ is a univalent function with $f(0)=0$ and $f'(0)=1$. Cunningham’s theorem (Theorem \ref{Cunningham}) describes the optimal pair $(r,R)$ consisting of the radius $r$ of the largest centered Euclidean disk contained in $\Omega$ and the capacity $R = \mathrm{cap}(\overline{\Omega})$, in terms of a parameter $A>1$.  Let $t := \sqrt{A}$ and $s=\frac{t-1}{t+1}$ so that $t=\frac{1+s}{1-s}$, $1-\frac{1}{t^2}=\frac{4s}{(s+1)^2}$ and $t+\frac{1}{t}= \frac{2(1+s^2)}{1-s^2}$. Then, $0<s<1$ and 
$$r= \frac{A-1}{4A} \cdot 
(\frac{\sqrt{A}+1}{\sqrt{A}-1})^{1/\sqrt{A}}=\frac{1}{4}(1-\frac{1}{t^2}) \cdot 
(\frac{t+1}{t-1})^{\frac{1}{t}}=\frac{s^{ \frac{2s}{1+s}}}{(1 + s)^2}$$
and 
$$R = \frac{(A-1)^2}{16A} \cdot (\frac{\sqrt{A}+1}{\sqrt{A}-1})^{1/\sqrt{A} + \sqrt{A}}=\frac{1}{16}(1-\frac{1}{t^2})^2 \cdot t^2 \cdot (\frac{t+1}{t-1})^{\frac{1}{t} + t}=\frac{s^{\frac{-4s^2}{1-s^2}}}{(1-s^2)^2}.$$

Notice that $r$ decreases with $s$. This follows by differentiating $\ln r(s) = (2 s/(1 + s)) \ln s - 2 \ln(1 + s)$, which gives $d/ds\ln r(s) = 2 \ln s/(1 + s)^2 < 0$ for $s \in (0,1)$, hence $r'(s) < 0$. Also, $r(s)\to 1$ as $s\to 0+$ and $r(s)\to 1/4$ as $s\to 1$.Similarly, one can check that $R$ strictly increases with $s$, $R(s)\to 1$ as $s\to 0+$ and $R(s)\to \infty$ as $s\to 1$.\\

Apply P\'{o}lya’s result (Proposition \ref{Polya}) to $f$, one has $1< R(s)=\mathrm{cap}(\overline{\Omega})=\mathrm{cap}(\overline{f(\mathbb{D})}) \leq \mathrm{cap}(P^{-1}(\overline{\mathbb{D}})) = c^{-1/d}$. Since $R$ strictly increases with $s$ and $R(s)\le C:=c^{-1/d}$,  there exists some $s_{\max}(C)$ such that $s \leq s_{\max}(C)$. For example, one can take $s_{\max}(C)$ to be $R^{-1}(C)$. Since $r$ decreases with $s$, the inclusion $\Omega \supset \mathbb{D}(0, r(s))$ improves as $s$ decreases. Therefore, if we can produce an explicit upper bound $s \leq s_{\max}(C)$ depending only on the capacity bound $C$, then we will have $r(s) \geq r(s_{\max}(C))$. 
Since $f(\mathbb{D})=D(P,1)$ contains the disk $\mathbb{D}(0,r(s_{\max}(C)))$, then for every critical point $b$ on $\partial D(P,1)$ $$S(P,b) = |P(b)| / (|b| |P'(0)|) = 1 / (|b|\cdot 1) \leq 1/r(s_{\max}(C))$$ Taking the maximum (of $S(P,b)$) over all critical points $b \in \partial D(P,1) \cap Z(P')$ gives
\begin{equation}\label{smax}
U(P) \leq \frac{1}{r(s)} \le \frac{1}{r(s_{\max}(C))} 
\end{equation}

To produce such an explicit upper bound $s_{\max}(C)$, we first rewrite $R$ as a function in $\lambda$ where $e^{-\lambda} = s^2$, i.e.,
$R=\exp( 2 \lambda/(e^{\lambda}-1))/(1-e^{-\lambda})^2$ which is a decreasing function on $(0,\infty)$. Hence, there is a unique parameter $\lambda_C$ such that $R(\lambda_C) = C$. Using the elementary inequality $e^{\lambda} - 1 \leq \lambda e^{\lambda}$ for $\lambda \geq 0$ (equivalently $1 - e^{-\lambda} \leq \lambda$), we obtain a lower bound
$R(\lambda) \geq \exp(2e^{-\lambda})/(1-e^{-\lambda})^2$ and hence $$C\ge \frac{\exp(2e^{-\lambda_C})}{(1-e^{-\lambda_C})^2}.$$\\
Taking square root gives
$$\sqrt{C}(1-e^{-\lambda_C})\ge \exp(e^{-\lambda_C})=e\exp(-(1-e^{-\lambda_C}))$$
and hence $(1-e^{-\lambda_C})\exp(1-e^{-\lambda_C})\ge \frac{e}{\sqrt{C}}$.
Since the principal Lambert function $W_0$ (the inverse of $w \mapsto w e^w$ on $[0, \infty)$) is strictly increasing, this means
$1-e^{-\lambda_C} \geq W_0(e/\sqrt{C})$. Recall that  $s^2=e^{-\lambda}$, we then conclude that any admissible $s$ must satisfy
$s \leq s_{*} := \sqrt{1 - W_0(e/\sqrt{C})} = \sqrt{1 - W_0(e c^{1/(2d)})}$. Because $r$ decreases with $s$, it follows that
$r(s) \geq r(s_{*}) = s_{*}^{ 2 s_{*}/(1 + s_{*}) } / (1 + s_{*})^2$, and therefore
$U(P) \leq 1/r(s) \leq (1 + s_{*})^2 / s_{*}^{ 2 s_{*}/(1 + s_{*}) }$. This proves i).\\

To prove ii), we simply drop the exponential factor in $R(s)$. Since $s \in (0,1)$ implies $s^{ - 4 s^2/(1 - s^2) } \geq 1$, we have $R(s) \geq 1/(1 - s^2)^2$. If $R(s) \leq C$, then $1 - s^2 \geq C^{-1/2}$, i.e. $s \leq \sqrt{1 - C^{-1/2}} = \sqrt{1 - c^{1/(2d)}}$. Calling this upper bound $\hat{s}$, the same monotonicity argument gives $r(s) \geq r(\hat{s})$, and by (\ref{smax}), we have 
$U(P) \leq g(\hat{s})$ with $\hat{s} = \sqrt{1 - c^{1/(2d)}}$.\\

For iii), as $c \to 0$ (so $C \to \infty$) one has $W_0(e/\sqrt{C}) \to 0$ and $s_{*} \to 1-$, hence $r(s_{*}) \to 1/4$ and $g(s_{\ast}) \to 4$. \\

To prove iv) and v), we first show that under some conditions, 
$$1 - r(s) \leq \sqrt{ (\ln C) \ln(2e/\ln C) }.$$

Recall that the capacity $R$ of $\overline{\Omega}=\overline{f(\mathbb{D})}=\overline{D(P,1)}$ can be expressed as $$R=\exp( 2 \lambda/(e^{\lambda}-1))/(1-e^{-\lambda})^2$$ which is strictly decreasing in $\lambda >0$ and $\lambda_{C}$ is defined by $$R(\lambda_C) = C = c^{-1/d}>1$$ Since $\mathrm{cap}(\overline{\Omega}) \le C$, we have $\lambda \ge \lambda_C$.

By the elementary bounds $1/(e^\lambda - 1) \geq e^{-\lambda}$ and $-\ln(1 - y) \geq y$ ($0 < y < 1$), we have
  \begin{align*}
  \ln R(\lambda) &= 2 \lambda/(e^\lambda - 1) - 2 \ln(1 - e^{-\lambda}) \\
          &\geq 2 \lambda e^{-\lambda} + 2 e^{-\lambda} \\
          &= 2 e^{-\lambda} (1 + \lambda).
  \end{align*}
Hence, at $\lambda = \lambda_C$ this gives
\begin{equation}\label{L_C}
2 e^{-\lambda_C} (1 + \lambda_C) \leq \ln C = L_c/d.
\end{equation}

Notice that the function $h(x) := (1 + x) e^{-(1 + x)}$ is strictly decreasing on $[0, \infty)$, with range $(0, 1/e]$. If $\ln C \leq 2$, then there exists a unique $\lambda^* \geq 0$ such that 
\begin{equation}\label{lambda*}
    (1 + \lambda^*) e^{-(1 + \lambda^*)} = (\ln C)/(2e) = (L_c/d)/(2e)
\end{equation}

From (\ref{L_C}) and the monotonicity of $h$ we get $\lambda_C \geq \lambda^*$. Since $\lambda \geq \lambda_C$ for every $\overline{\Omega}=\overline{f(\mathbb{D})}=\overline{D(P,1)}$, we have $\lambda \geq \lambda^*$.

Now we try to bound $1 - r(s)$ in terms of $\lambda$. Recall that $r(s) = s^{2s/(1+s)}/(1 + s)^2$ and $\ln s = -\lambda/2$. We have
  \begin{align*}
  1 - r(s) &= 1 - (1 + s)^{-2} + (1 + s)^{-2} [1 - s^{2s/(1+s)}] \\
           &\leq 2s/(1 + s) + (2s/(1 + s)) \cdot (\lambda/2) \\
           &\leq s (2 + \lambda)=e^{-\lambda/2}(2 + \lambda). 
  \end{align*}
  Here we used $1 - (1 + s)^{-2} \leq 2s/(1 + s)$ and $1 - e^{-x} \leq x$ with $x = (2s/(1 + s))\cdot(-\ln s) = (2s/(1 + s))\cdot(\lambda/2)$. Since $s = e^{-\lambda/2}$ and $q(\lambda) := e^{-\lambda/2} (2 + \lambda)$ has derivative $q'(\lambda) = -(\lambda/2) e^{-\lambda/2} \leq 0$, $q$ is decreasing on $[0, \infty)$. Using $\lambda \geq \lambda^*$, we have
  \begin{equation}\label{r}
  1 - r(s) \leq e^{-\lambda^*/2} (2 + \lambda^*).
\end{equation}

Multiplying both sides of (\ref{lambda*}) by $e$ gives $e^{-\lambda^*} (1 + \lambda^*) = (\ln C)/2$. Hence
  $e^{-\lambda^*/2} = \sqrt[ ]{ (\ln C) / (2(1 + \lambda^*)) }$.
  Plugging this into (\ref{r}), we have
\begin{equation}\label{r2}
1 - r(s) \leq (2 + \lambda^*) \sqrt[ ]{ (\ln C) / (2(1 + \lambda^*)) }.  
\end{equation} 
We now compare $(2 + \lambda^*)$ with $(1 + \lambda^*) - \ln(1 + \lambda^*)$. For $y := 1 + \lambda^* \geq 6$ one can check that 
\begin{equation}\label{y}
(y + 1)^2 \leq 2y(y - \ln y) 
\end{equation}
Indeed, $F(y) := 2y(y - \ln y) - (y + 1)^2$ has $F'(y) = 2(y - \ln y - 2) \geq 0$ for $y \geq 6$ and $F(6) > 0$. 

Setting $y = 1 + \lambda^*$ in (\ref{y}) will give
  $$(2 + \lambda^*)^2 \leq 2(1 + \lambda^*)[(1 + \lambda^*) - \ln(1 + \lambda^*)].$$
  Taking square roots and inserting into (\ref{r2}),
  $1 - r(s) \leq \sqrt{ (\ln C) \cdot ((1 + \lambda^*) - \ln(1 + \lambda^*)) }$.

Taking logs in (\ref{lambda*}) yields $(1 + \lambda^*) - \ln(1 + \lambda^*) = \ln(2e/\ln C)$. Therefore, 

\begin{equation}\label{r3}
1 - r(s) \leq \sqrt{ (\ln C) \ln(2e/\ln C) }
\end{equation}
provided that $y=1 + \lambda^* \geq 6$.\\

It remains to check that $\ln C \le 2$ and $y=1 + \lambda^* \geq 6$ for the ranges for $d$ assumed in iv) and v).

For iv) , we assume $d \geq 15 L_c$ and hence $\ln C  \leq 1/15\le 2$. From (\ref{lambda*}), $y e^{-y} = (\ln C)/(2e)$ Therefore, $(\ln C)/(2e) \leq 1/(30e) \approx 0.01226$, where $6 e^{-6} \approx 0.01487$. Since $ye^{-y}$ is a decreasing function  on $[1, \infty)$, we have $y \geq 6$.\\

For v), we assume $d \geq \max\{L_c^2, 169\}$, so $\ln C \leq 1/\sqrt{d} \leq 1/\sqrt{169}\approx 0.07692 \le 2$. Therefore, $(\ln C)/(2e) \leq 1/(26e) \approx 0.01415 < 0.0149$. We can again conclude that $y \geq 6$.\\

Furthermore, one can check that $\sqrt{(\ln C) \ln(2e/\ln C)} < 2/3$ for the ranges of $d$ assumed in iv) and v).
Since  (\ref{r3}) gives $r(s)\geq 1-\sqrt{(\ln C) \ln(2e/\ln C)}$, we can use $1/(1 - x) \leq 1 + 3x$ for $x \leq 2/3$ to get
  \begin{align*}
  1/r(s) &\leq 1 + 3 \sqrt{ (\ln C) \ln(2e/\ln C) } = 1 + 3 \sqrt{ (L_c/d) \ln((2ed)/L_c) }. 
  \end{align*}

Therefore, if $d \geq  15 L_c$ or $d \geq \max\{L_c^2, 169\}$, then from $U(P) \le 1/r(s)$, we have
\begin{equation}\label{S1}
  U(P) \leq 1 + 3 \sqrt{ (\ln C) \ln(2e/\ln C) } = 1 + 3 \sqrt{ \frac{L_c\ln((2ed)/L_c)}{d}}. 
\end{equation}

For v), since $d \geq L_c^2$, $L_c \leq \sqrt{d}$ and $\ln C = L_c/d \leq 1/\sqrt{d}$. One can check that the function $x \mapsto x \ln(2e/x)$ is increasing on $(0,1]$ and hence
  $\sqrt{ (\ln C)\ln(2e/\ln C)} \leq \sqrt{ (1/\sqrt{d}) \ln(2e\sqrt{d})}= d^{-1/4} \sqrt{ \ln(2e\sqrt{d}) }$.

From (\ref{S1}),
   \begin{equation}
  U(P) \leq  1 + 3 d^{-1/4} \sqrt{\ln(2e\sqrt{d})}\,\to 1,\quad \mathrm{as}\quad d \to \infty.
  \end{equation}
       
In particular, if $c \geq e^{-\sqrt{d}}$ and $d \geq 169$, then $d \geq \max\{L_c^2, 169\}$ and 
$$3 d^{-1/4} \sqrt{\ln(2e\sqrt{d})} \le 3 d_0^{-1/4} \sqrt{\ln(2e\sqrt{d_0})}\approx 1.7169\ldots$$ where $d_0=169$. Therefore, $U(P) \leq 2.7169\ldots$ and hence $S(P,b) \leq 2.7169\ldots$ for every $b \in Z(P') \cap \partial D(P,1)$\\

\noindent
{\it Proof of Theorem \ref{thm:2}.} Like Theorem \ref{thm:1}, we will also prove that Theorem \ref{thm:2} still holds when $S(P)$ is replaced by $U(P)$. We first prove that for any $R>1$, 
\begin{equation}\label{C}
U(P) \leq 2 + \log(A_d/R).
\end{equation}

By the Markov-type inequality of Eremenko, we have for a degree-$d$ polynomial $P$ and $\Omega=D(P,1)$,
\begin{equation}\label{P'}
    \sup_{z\in \overline{\Omega}} |P'(z)| \cdot \mathrm{cap}(\overline{\Omega}) \leq A_d
\end{equation}

Passing to $f=P^{-1}=z+a_2z^2+\cdots$ on $\mathbb{D}$ gives $|f'(\zeta)| = 1/|P'(f(\zeta))|\geq R/A_d$ for all $ \zeta\in\mathbb{D}$. Since $f'$ has no zeros on $\mathbb{D}$, we can define $g:=\log f'$ holomorphic on $\mathbb{D}$. Then $\operatorname{Re} g(\zeta) = \log|f'(\zeta)|\geq\log(R/A_d) $ for all $\zeta\in\mathbb{D}$.\\

Let $ \phi(\zeta) := g(\zeta) - \log(R/A_d) $. Then $ \phi $ is holomorphic on $\mathbb{D}$, $ \operatorname{Re} \phi \geq 0 $, and $ \phi(0)= -\log(R/A_d) \geq 0 $ because of (\ref{ieq:R}). Carathéodory’s lemma (for functions with nonnegative real part, see page 41, section 2.5 of Duren's book \cite{Duren1983}) gives $ |\phi'(0)| \leq 2 \operatorname{Re} \phi(0) = -2 \log(R/A_d)$. But $ \phi'(0) = g'(0) = f''(0)/f'(0) = 2 a_2 $, so $$ |a_2| \leq -\log(R/A_d) = \log(A_d/R)$$\\

The usual Koebe-quarter type inequality for Schlicht functions gives $ r \geq 1/(2+|a_2|)$ where $r$ is the largest radius of a disk centered at $0$ and lying inside $\Omega$. Therefore
$$U(P) \leq 1/r \leq 2 + |a_2| \leq 2 + \log(A_d/R).$$

We now prove that \begin{equation}\label{M}
U(P) \leq A_d/R.
\end{equation}

From (\ref{P'}), we have $\sup_{\Omega} |P'| \leq A_d/R$. Let $a$ be a critical point of $P$ which is also a boundary point of $\Omega$. Let $\alpha$ be a boundary point of $\Omega$ so that $|\alpha|=\text{dist}(0, \partial\Omega)$. Integrating $|P'|$ along a straight line in $\Omega$ joining $0$ to $\alpha \in \partial\Omega$ yields $1=|P(\alpha)|=|P(\alpha)-P(0)|\le \delta\frac{ A_d}{R}$ where $\delta := \text{dist}(0, \partial\Omega)$. Since $|a|\ge \delta$, we have 
$$U(P) \leq 1/|a| \leq 1/\delta \le A_d/R.$$

Finally, we prove that for any $R>1$,
\begin{equation}\label{DS}
U(P)\le 4-\frac{e^2}{4R}
\end{equation}

 Notice that when $R > \frac{e^2 \pi^2}{64} $, (\ref{DS}) follows easily from a result of Duren and Schiffer \cite{DurenSchiffer1991} which says that if $f:\mathbb{D}\to f(\mathbb{D})$ is a univalent function $f(z)=z+a_2z^2+\cdots$ with $$\mathrm{cap}(\overline{f(\mathbb{D})}) > \frac{e^2 \pi^2}{64}$$ then $$|a_2| \leq 2 - \frac{e^2}{4\mathrm{cap}(\overline{f(\mathbb{D})})}$$ and the inequality is sharp. Indeed when $R=\mathrm{cap}(\overline{\Omega}) > \frac{e^2 \pi^2}{64}$, then 
it follows from the proof of the Koebe $\frac{1}{4}$-Theorem that $$S(P,b) \leq 2+ |a_2|\leq 4 - \frac{e^2}{4R}$$ for any critical point $b \in \partial D(P,t(P))$.

Here we will prove (\ref{DS}) for all $R>1$ using Cunningham's Theorem (Theorem \ref{Cunningham}).\\

  We first get a hyperbolic parametrization for the functions $r$ and $R$ in (\ref{rA}) and (\ref{RA}). For $A>0$, define $\rho := \arccoth(\sqrt{A}) = (1/2) \ln((\sqrt{A} + 1)/(\sqrt{A} - 1)) > 0$, so that
$\sqrt{A} = \coth \rho$, $1/\sqrt{A} = \tanh \rho$, and $(\sqrt{A} + 1)/(\sqrt{A} - 1) = e^{2\rho}$. Then $(A - 1)/A = \sech^2 \rho$ and $(A - 1)^2/A = \csch^2 \rho \cdot \sech^2 \rho$. We then rewrite (\ref{rA}) and (\ref{RA}) as
\begin{equation}
r(\rho) = \frac{1}{4} \sech^2 \rho \cdot e^{2\rho \tanh \rho}
\end{equation}
and
\begin{equation}
R(\rho) = \frac{1}{16} \sech^2 \rho \cdot \csch^2 \rho \cdot e^{2\rho (\tanh \rho + \coth \rho)}
\end{equation}

  For $\rho>0$, define
$X(\rho):= \sech^2\rho \cdot e^{2\rho \tanh \rho}$, $Y(\rho):= \csch^2\rho \cdot e^{2\rho \coth \rho}. $Then one have $r(\rho) = X(\rho)/4$, $R(\rho) = X(\rho) Y(\rho)/16$ and $Y = 4R/r$. Let $Z(\rho):= Y(\rho) [X(\rho) - 1] = YX - Y = 16R - Y$. We first show that $Z$ is increasing and $e^2 \le Z(\rho) \leq 12$ for all $\rho>0$.\\

Notice that the logarithmic derivatives of $X$ and $Y$ are $2\rho \sech^2 \rho$ and $-2\rho \csch^2 \rho$ respectively. Differentiate $Z = Y(X - 1)$ and we then have $Z' = Y'(X - 1) + YX'= Y(-2\rho \csch^2 \rho (X - 1) + 2\rho \sech^2 \rho X)$. Hence $Z' \geq 0$ is equivalent to
$$X \sech^2 \rho \geq (X - 1) \csch^2 \rho \Leftrightarrow \cosh^2 \rho \geq X.$$

Since $X = e^{2\rho \tanh \rho} \sech^2 \rho$, the inequality $\cosh^2 \rho \geq X$ is the same as $e^{2\rho \tanh \rho} \leq \cosh^4 \rho$.
Set $g(\rho) := 2\rho \tanh \rho - 4 \ln \cosh \rho$. Then $g'(\rho) = 2(\rho \sech^2 \rho - \tanh \rho) \leq 0$ because $\tanh \rho = \int_0^\rho \sech^2 t \, dt \geq \rho \cdot \sech^2 \rho$ (for $\sech^2$ is decreasing on $(0,\infty)$). With $g(0) = 0$ we get $g(\rho) \leq 0$ for all $\rho > 0$. So we conclude that $e^{2\rho \tanh \rho} \leq \cosh^4 \rho$ and $Z' \geq 0$. Finally, as $\rho \to 0+$: $\tanh \rho = \rho + O(\rho^3)$, $\coth \rho = 1/\rho + O(\rho)$, so
  $X = 1 + \rho^2 + O(\rho^4)$, $Y = e^2 \rho^{-2} + O(1)$, hence
  $Z = Y(X - 1) = e^2 + O(\rho^2) \to e^2$. On the other hand, as $\rho \to \infty$: $X \to 4$ and $Y \to 4$, hence $Z \to 4\cdot(4 - 1) = 12$.\\

We now show that $1/r \leq 4 - e^2/(4R)$. From $Z = 16R - Y$ and $Z \geq e^2$ we obtain $Y \leq 16R - e^2$. Since $Y=4R/r$, $4R/r \leq 16R - e^2$
  implies $1/r \leq 4 - e^2/(4R)$ for every $\rho > 0$, hence for every $R > 1$. Notice that Cunningham’s theorem gives a unique $A > 1$, hence a unique $\rho > 0$, for each $R > 1$. To see this, we notice that
$$\frac{d}{d\rho}  \ln R = \frac{d}{d \rho}(\ln X + \ln Y - \ln 16)= 2\rho(\sech^2 \rho - \csch^2 \rho) < 0$$ Thus $R(\rho)$ decreases strictly from $\infty$ (as $\rho \to 0+$) to 1 (as $\rho \to \infty$), so the hyperbolic parametrization covers exactly $R > 1$. \\

Now, for $1<R\le 1.8$, $U(P) \le 1/r \leq 4 - e^2/(4R)\le 4 - e^2/(4\times 1.8)=2.9737422...\le 3$.

Also, for $R\ge A_d/3$, $U(P)\le A_d/R \le A_d/(A_d/3)=3$ and we complete the proof.\\

\section{Extremal polynomials for $U(P)$}\label{sec3}

It is known by work of Crane \cite{Crane2006} (see also Ng \cite{Ng}), that there exist extremal polynomials for Smale's mean value conjecture that satisfy a certain special property, that is, if $\mathcal{P}_{d}$ denotes the class of polynomials $P$ of degrees between $2$ and $d \geq 2$ inclusive, satisfying $P(0) = 0$, $P'(0) \neq 0$, then:
$$\sup_{P \in \mathcal{P}_{d}} S(P) = S(P_{\infty})$$
for some polynomial $P_{\infty} \in \mathcal{P}_{d}$ of degree $d$ which further satisfies the following special property:

    For any critical points $b,b' \in Z(P_{\infty}')$, one has $$S(P_{\infty},b) = S(P_{\infty},b')$$

Such extremal polynomials are called ``Standard extremal polynomial'' (by Crane in \cite{Crane2006}). Concerning the quantity $\sup_{P \in \mathcal{P}_{d}} U(P)$, so far there is no available information about the nature of extremal polynomials.  However, $\sup_{P \in \mathcal{P}_{d}} U(P) \ge \frac{d-2}{2d}(\sqrt{2}+1)$ for even $d \geq 14$ as this is evidenced by polynomials of the form $P_m = \frac{m}{m-1}(z - \frac{1}{m}z^m) \circ (4z^2 - 4z)$ for all $m \ge 2$. This fact was first observed by Dubinin in \cite{dubinin2012some} who considered the case when $m=7$. Note that for $m\ge 2$, all $P_m$ have critical values of equal modulus $1 = t(P_m)$ when $m$ is odd and when $m$ is even, all but one of the critical values are on the unit circle (the exceptional critical value has modulus $\frac{m+1}{m-1}$). 

Notice that $\frac{d-2}{2d}(\sqrt{2}+1)>\frac{d-1}{d}$ for $d\ge 4$. It follows then that, if Smale's mean value conjecture is true, Crane's ``standard extremal polynomials'' cannot also be extremal polynomials for $U(P)$. Therefore the following theorem appears to be meaningful.

\begin{theorem} \label{thm:9}
There exists an extremal polynomial $Q_{\infty}$ for $\sup_{P \in \mathcal{P}_{d}} U(P)$, which satisfies the following property:
$$|Q_{\infty}(b)| = |Q_{\infty}(b')| $$ for any two critical points $b,b' \in Z(Q_{\infty}') \setminus Z(Q_{\infty})$. That is, $Q_{\infty}$'s non-zero critical values lie on the same circle center zero.

\end{theorem}
\begin{remark}
The result of this section is logically independent of the main results of all other sections. The proof technique is, however, reused (and made reference to) in the proofs of Theorem \ref{thm:3} and Corollary \ref{cor:2} in Section \ref{sec:4}.
\end{remark}
\medskip  
\subsection{Proof of Theorem \ref{thm:9}}
\begin{proof}[Proof of Theorem \ref{thm:9}]
We start by showing that extremal polynomials for $\sup_{P \in \mathcal{P}_d} U(P)$ exist. Let $$P_n ; n=1,2, \ldots $$ be an extremizing sequence of polynomials in $\mathcal{P}_{d}$, i.e, a sequence of polynomials in $\mathcal{P}_d$ such that $$U(P_n) \rightarrow \sup_{P \in \mathcal{P}_{d} } U(P) $$ which we renormalize so that $P_n'(0) = 1$, and $t(P_n) = 1$ for all $n$. Let $$b_n \in Z(P_n') \cap \partial D(P_n,1) ; n=1,2, \ldots$$ be such that $$U(P_n) = \frac{1}{|b_n|} ; n=1,2, \ldots $$ Then $$|b_n| ; n=1,2, \ldots$$ is bounded (above and below uniformly in $n$), since the Koebe-$\frac{1}{4}$-theorem implies that it is bounded below by $\frac{1}{4}$, and there exists no subsequence $n_{k} ; k = 1,2 \ldots $ of $n ; n=1,2, \ldots$ for which $$|b_{n_{k}}| \rightarrow \infty \quad \rm{as} \qquad k \rightarrow \infty$$ or else, if $n_{k} ; k=1,2,...$ was such a subsequence, then $$\sup_{P \in \mathcal{P}_{d} } U(P) = \lim_{k \rightarrow \infty} U(P_{n_{k}}) = \lim_{k \rightarrow \infty} \frac{1}{|b_{n_{k}}|} = 0$$ which is a contradiction since $$U(z^d - z) = 1 - \frac{1}{d} \geq \frac{1}{2} ; d \geq 2$$

By the compactness of Schlicht normalized polynomials of degrees $\leq d$ (see Remark \ref{compactnessOfSchlichtNormalizedPolynomials}) we may without loss of generality, after passing to an appropriate subsequence whose relabeling we suppress, assume that $P_n \rightarrow Q$ holds locally uniformly in $\mathbb{C}$, where by Lemma \ref{lem:10} of Section \ref{sec:4} we have that $Q \in \mathcal{P}_{d}$ is also Schlicht normalized. Then (after passing to a further subsequence, if necessary, whose relabeling we suppress), $$S(P_n,b_n) \rightarrow S(Q,b)$$ for some critical point $b = \lim_{n} b_n$ of $Q$, which is necessarily on $$\partial D(Q,t(Q)) \equiv \partial D(Q,1)$$ again by Lemma \ref{lem:10} of Section \ref{sec:4}.

Since $b \in \partial D(Q,1) \cap Z(Q')$, we have by definition of $U(Q)$ as $$\max_{b \in \partial D(Q,1) \cap Z(Q')} S(Q,b)$$ that \begin{equation} \label{firstineqforU} \lim_n S(P_n,b_n ) = S(Q,b) \leq U(Q) \end{equation} On the other hand \begin{equation} \label{byDefOfU} \lim_{n} U(P_n) = \lim_{n} S(P_n,b_n)= \sup_{P \in \mathcal{P}_{d}} U(P)  \end{equation} by construction, so combining (\ref{firstineqforU}) and (\ref{byDefOfU}) we conclude that $$U(Q) \leq \sup_{P \in \mathcal{P}_{d}} U(P) = \lim_{n} S(P_n,b_n) = S(Q,b) \leq U(Q) $$

It follows that $Q$ is a (Schlicht normalized) extremal polynomial for $\sup_{P \in \mathcal{P}_d} U(P)$ over $\mathcal{P}_{d}$.\\

To see that there exists an extremal polynomial for $U(.)$ over $\mathcal{P}_{d}$ having all critical values of the same modulus, we apply the quasiconformal deformation argument of Eremenko and Hayman, essentially verbatim (see \cite{eremenko1999length} for details).

Let $P$ be a Schlicht normalized extremal polynomial for $U(.)$ over $\mathcal{P}_{d}$ and suppose that $a \neq 0$ is a critical value of $P$ satisfying $a \notin \partial \mathbb{D}$. Define $\psi_{\lambda}(z) = z + \lambda \Phi(z)$, where $\Phi$ is a smooth compactly supported  function satisfying $$\Phi(a) = 1$$ and $$\mathrm{\supp{(\Phi)}} \subset \Delta(a ; \epsilon)$$ where $$\Delta(a ; \epsilon)$$ is a small open disk centered at $a$ which lies a fixed positive distance away from all other critical values of $P$, the origin and $\partial \mathbb{D}$, i.e. $$\mathrm{dist}(\Delta(a ; \epsilon) , \{ 0 \} \cup [\{P(b) : b \in Z(P') \} \setminus \{a \}] \cup \partial \mathbb{D}) \eqcolon \delta > 0$$ We further suppose that $|\lambda|$ is sufficiently  small (say $\lambda \in N$, where $0 \in N$ is a fixed sufficiently small open disk centered at the origin) so that $$\psi_{\lambda}$$ is a quasiconformal homeomorphism of the complex plane.
Then define the deformed polynomial $$P_{\lambda} \coloneq \psi_{\lambda} \circ P \circ \phi_{\lambda}^{-1}$$ where $\phi_{\lambda}$ is the quasiconformal mapping of the complex plane fixing $0$ and $\infty$, having local form $\phi_{\lambda}(z) = z + o(1)$ near $\infty$ for each fixed $\lambda$, and with Beltrami coefficient $$\mu_{\lambda} = (\mu_{\psi_{\lambda}} \circ P) \ \frac{\overline{P'}}{P'}$$ It follows  that the resulting map $P_{\lambda}$ is conformal (and thus necessarily a polynomial of the same degree as $P$). Such a quasiconformal mapping exists by the measurable Riemann mapping theorem, which also implies that \begin{itemize}
    \item[i)]  For any fixed $z \in \mathbb{C}$, $$\lambda \rightarrow \phi_{\lambda}(z)$$ is holomorphic 
    \item[ii)]  $$(\lambda,z) \rightarrow \phi_{\lambda}(z) : N \times (\mathbb{C} \setminus P^{-1}(\mathrm{\supp}(\Phi))) \rightarrow \mathbb{C}$$ is (jointly) holomorphic in both variables $\lambda$ and $z$.
\end{itemize} For a proof of this fact, see e.g. Theorem 5.7.4 on page 188 of \cite{AstalaIwaniecMartin}. It follows that for any fixed $z \in \mathbb{C} \setminus P^{-1}(\mathrm{\supp}(\Phi))$, the mapping $$\lambda \rightarrow \phi_{\lambda}'(z) = \frac{\partial}{\partial z} \phi_{\lambda}(z) : N \rightarrow \mathbb{C}$$ is holomorphic. Hence the mapping $$\lambda \rightarrow \phi_{\lambda}'(0) : N \rightarrow \mathbb{C}$$ is holomorphic. Therefore after applying the rescaling $$\phi_{\lambda} \rightarrow \frac{\phi_{\lambda}(z)}{\phi_{\lambda}'(0)}$$ the mapping $$\lambda \rightarrow \phi_{\lambda}(z) : N \rightarrow \mathbb{C}$$ is still holomorphic for any fixed $z \in \mathbb{C}$ and the above two holomorphicity properties (i) and (ii) continue to hold. We call the rescaled quasiconformal homeomorphism again $\phi_{\lambda}$, and the corresponding deformed polynomial again $P_{\lambda}$. Note that by construction $$P_{\lambda}'(0) = 1 ; \lambda \in N$$ and $$P_{\lambda}(0) = 0 ; \lambda \in N$$

Since $\partial D(P,1)$ is a Jordan curve in the complex plane (this follows from Lemma \ref{lem:10}), $\phi_{\lambda}(\partial D(P,1))$ is a Jordan curve in the complex plane. $\phi_{\lambda}(\partial D(P,1))$ winds once positively around the origin because $\phi_{\lambda}$ fixes the origin and infinity, and is an orientation preserving homeomorphism of the complex plane. Moreover, $\phi_{\lambda}(\partial D(P,1))$ contains no zeros of $P_{\lambda}$ in its interior besides $\phi_{\lambda}(0) = 0$ because the zeros of $P_{\lambda}$ are precisely of the form $$\phi_{\lambda}(z) ; z \in Z(P)$$ Letting $$z(s) \coloneq P^{-1}(\exp(is))$$ where by $P^{-1}$ we mean the continuous extension (to $\overline{\mathbb{D}}$) of the inverse branch of $P^{-1}$ fixing the origin, we have $$P_{\lambda}(\phi_{\lambda}(z(s))) = \psi_{\lambda}(P(P^{-1}(\exp(is)))) = \exp(is)$$ because $\psi_{\lambda}$ is the identity mapping in a neighbourhood of the unit circle. Since $\phi_{\lambda}(D(P,1))$ is a Jordan domain containing exactly one zero (counting multiplicity) of $P_{\lambda}$, which is at the origin $$\restr{P_{\lambda} \ }{\phi_{\lambda}(D(P,1))}$$ is a univalent map of $\phi_{\lambda}(D(P,1))$ onto $\mathbb{D}$, so that $$D(P_{\lambda},1) = \phi_{\lambda}(D(P,1))$$ and $$t(P_{\lambda}) \geq 1$$ moreover $$\partial D(P_{\lambda},1) = \phi_{\lambda}(\partial D(P,1)) ; \lambda \in N$$ Since $$\partial D(P_{\lambda},1) = \phi_{\lambda}(\partial D(P,1))$$ contains at least one critical point of $P_{\lambda}$, because the critical points of $P_{\lambda}$ are precisely given by $$\phi_{\lambda}(b) ; b\in Z(P')$$ and there is at least one critical point of $P$ on $\partial D(P,1)$ as $P$ is Schlicht normalized, we conclude that there is at least one critical point of $P_{\lambda}$ on $$\partial D(P_{\lambda},1)$$ so that $$t(P_{\lambda}) \leq 1$$ hence $$t(P_{\lambda}) = 1$$ Thus, the deformed polynomials $$P_{\lambda} :\lambda \in N$$ are Schlicht normalized and $$ \partial D(P_{\lambda},t(P_{\lambda})) = \partial D(P_{\lambda},1) = \phi_{\lambda}(\partial D(P,1)) ; \lambda \in N$$ Consequently, the function $$\lambda \rightarrow U(P_{\lambda}) = \max_{ b_{\lambda} \in \partial D(P_{\lambda},1) \cap Z(P_{\lambda}')} \frac{1}{|b_{\lambda}|} = \max_{ b \in \partial D(P,1) \cap Z(P')} \frac{1}{|\phi_{\lambda}(b)|}$$ is subharmonic in $\lambda \in N$ (as it is the maximum of finitely many subharmonic functions) and achieves a maximum at $\lambda = 0$. By the maximum principle for subharmonic functions, we conclude that $$\lambda \rightarrow U(P_{\lambda})$$ is necessarily constant.

By a standard limiting argument (as in \cite{eremenko1999length}) we can move each of the (non-zero) critical values of modulus $a \neq 1 = t(P)$ towards the unit circle over a sequence of extremal polynomials constructed by the above quasiconformal deformation of $P$. In the limit, we obtain an extremal polynomial with all (non-zero) critical values of equal modulus $t(P) = 1$.

\end{proof}
\section{An improved Markov type inequality} \label{sec:4}
Let $\mathcal{P}_{d}$ denote the class of polynomials $P$ of degrees between $2$ and $d \geq 2$ inclusive, satisfying $P(0) = 0$, $P'(0) \neq 0$. The following result gives an improvement on the Markov inequality of Eremenko for the regions $D(P,t(P))$.

\improvedMarkov*

We first establish some preliminary results which will be used in the proof of Theorem \ref{thm:3}. In particular, we will establish a continuity result (Corollary \ref{cor:4}) which says that $$\mathrm{cap}(\overline{D(P_n,1)}) \rightarrow \mathrm{cap}(\overline{D(Q,1)})$$ as $n \rightarrow \infty$, given the polynomials $P_n$ are non-linear and Schlicht normalized, and converge locally uniformly to $Q$.

The same property is \textbf{not true} for an arbitrary sequence $f_n$ of Schlicht univalent functions converging locally uniformly to a Schlicht univalent function $f$, in the sense that $$\mathrm{cap}(\overline{f(\mathbb{D})}) < \liminf_n \mathrm{cap}(\overline{f_n(\mathbb{D})}) $$ is possible. Indeed, one can perform a so-called ``pinching'' construction to show this (see page 3 of \cite{DurenSchiffer1991}). The idea is most easily conveyed through a figure eight. One can trace a nested sequence of Jordan curves $\gamma_n$ lying exterior to a ``figure eight'' or ``bow-tie'' shaped curve, which successively approach the figure eight curve from the \textbf{outside}, and form a narrower and narrower ``neck'' region close to the double tangency point of the figure eight. The kernel of the images of the (normalized, with positive derivative at the origin) conformal mappings $f_n$ onto the inner domains of each $\gamma_n$ with respect to a fixed ``center point'' $c = f_n(0) , n=1,2,... $ of one of the lobes comprising the figure eight, for definiteness say that $c$ belongs to the right lobe,  will exclude the left lobe comprising the figure eight, hence the capacity of the image of the resulting limit mapping $f$ will be strictly smaller than the limit inferior of the capacities of the images of the maps in the sequence. Such a sequence of conformal mappings can be rescaled via $f_n \rightarrow \frac{f_n}{|f_n'(0)|}$ so that they each have derivative of unit magnitude at the origin. The qualitative behaviour of the images of the rescaled mappings will not change since $|f_n'(0)|$ is bounded below by the inner radius of the right lobe with respect to $c$ and is bounded above by the inner radius of the image of $f_1$ with respect to $c$. In particular, $$\liminf_n \mathrm{cap}{\overline{f_n(\mathbb{D})}} > \mathrm{cap}{\overline{f(\mathbb{D})}}$$ will continue to hold.

This highlights the necessity of establishing the following preliminary results.

\begin{lemma} \label{lem:10}
    Let $P_n \in \mathcal{P}_{d} , \ n=1,2,3,...$ be a sequence of Schlicht normalized polynomials, converging locally uniformly to the polynomial $Q$. Then $Q \in \mathcal{P}_{d}$ and is Schlicht normalized. Furthermore, letting $P_n^{-1}$ and $Q^{-1}$ denote the inverse branches fixing zero defined on the closed unit disk $\overline{\mathbb{D}}$ we have $$\restr{P_n^{-1}}{\overline{\mathbb{D}}} \rightarrow \restr{Q^{-1}}{\overline{\mathbb{D}}}$$ uniformly on $\overline{\mathbb{D}}$ as $n \rightarrow \infty$.

\end{lemma}

\begin{proof}[Proof of Lemma \ref{lem:10}]
 Clearly $Q(0) = 0$ and $Q'(0) = 1$ holds by the local-uniform convergence $P_{n} \rightarrow Q$. Let $$f_{n} ; n=1,2, \ldots$$ denote the inverse branch of $P_{n}^{-1}$ fixing zero defined on the open unit disk $\mathbb{D}$. By (\ref{EM}), we have \begin{equation} \label{markovBoundedness} |f_n(w)| = |f_n(w) - 0| \leq \mathrm{diam}(\overline{D(P_n,1)}) \leq 4 \mathrm{cap}(\overline{D(P_n,1)}) = 4 \mathrm{cap}(\overline{D(P_n,1)}) |P_n'(0)|  \end{equation} $$\leq 4 \mathrm{cap}(\overline{D(P_n,1)}) ||P_{n}'||_{\overline{D(P_n,1)}} \leq  2^{1+\frac{1}{d}} d^2 $$ for every $w \in \mathbb{D}$. It follows that $\{f_{n} \}$ is a normal family on $\mathbb{D}$. Let $h$ be any subsequential limit of $f_{n}$. Then $h$ is non-constant (since $h'(0) = 1$) and it follows by Hurwitz's theorem that $h : \mathbb{D} \rightarrow \mathbb{C}$ is Schlicht. Moreover for any $w \in \mathbb{D}$, by (\ref{markovBoundedness}), we have $$|h(w)| \leq  2^{1+\frac{1}{d}}d^2$$ Since $P_{n} \rightarrow Q$ holds locally uniformly in $\mathbb{C}$, we conclude that $$
Q(h(w))
=\lim_{k\to\infty}Q(f_{n_k}(w))
=\lim_{k\to\infty}P_{n_k}(f_{n_k}(w))
=\lim_{k\to\infty}w
=w.
$$
This implies in particular that $h$ is the (univalent) inverse branch of $Q$ fixing $0$, since this is defined on all of $\mathbb{D}$ we conclude that \begin{equation} \label{inequalityFort(Q)first} t(Q) \geq 1 \end{equation} and it follows that $f_{n} \rightarrow f$ locally uniformly on $\mathbb{D}$, where $f$ denotes the inverse branch of $Q^{-1}$ fixing zero defined on the open unit disk. Since $\partial D(P_n,1) ; n=1,2,...$ and $\partial D(Q,1)$ are (away from a finite set of critical points), real analytic curves and are thus locally-connected, the maps $f_{n}; n=1,2,....$ and $f$ extend continuously to $\overline{\mathbb{D}}$ by Carath\'eodory's theorem (Theorem 2.1 of page 20 of \cite{pommerenke1992boundary}). On the other hand we have $$Q(f(z)) = z ; z \in \overline{\mathbb{D}}$$ and $$P_n(f_n(z)) = z ; z \in \overline{\mathbb{D}} $$ by continuity (for all $n = 1,2, \ldots$), so that we conclude that $f$ and $f_n;n=1,2, \ldots$ are injective on $\overline{\mathbb{D}}$ and in particular $$\partial D(P_n,1) = f_n(\partial \mathbb{D})$$ and $$\partial D(Q,1) = f(\partial \mathbb{D})$$ are Jordan curves in the complex plane. That $f_n$ converges to $f$ uniformly on $\overline{\mathbb{D}}$ as $n \rightarrow \infty$ will follow from the following two statements.
\begin{enumerate}
\renewcommand{\labelenumi}{(\arabic{enumi})}
    \item \label{first} $$f_n\rightarrow f$$  holds pointwise everywhere on $\overline{\mathbb{D}}$ as $n \rightarrow \infty$.
    \item \label{second} For every $w \in \partial \mathbb{D}$ and $\epsilon > 0$ there exists a $\delta = \delta(\epsilon,w)$ such that for all $n$ sufficiently large, $$|\restr{f_n \ }{\partial \mathbb{D}}(w_1) - \restr{f_n \ }{\partial \mathbb{D}}(w_2)| < \epsilon$$ holds for every $w_1, w_2 \in \Delta(w ; \delta) \cap \partial \mathbb{D}$
    \end{enumerate} 

In fact, each $f_n-f$ is holomorphic on $\mathbb{D}$ and continuous on $\overline{\mathbb{D}}$, so by the maximum principle, we have $$\max_{\overline{\mathbb{D}}}|f_n-f|=\max_{\partial{\mathbb{D}}}|f_n-f|$$
Thus, to prove uniform convergence of $\{f_n\}$ to $f$ on $\overline{\mathbb{D}}$, it suffices to prove uniform convergence on $\partial \mathbb{D}$, which follows from statements (\ref{first}) and (\ref{second}) by the compactness of $\partial\mathbb{D}$.

We first prove statement (\ref{second}) and also the open disk part of statement (\ref{first}). By (\ref{markovBoundedness}), we have $$f_n(\overline{\mathbb{D}}) \subset \overline{\mathbb{D}}_{R_0} ; n = 1 ,2, \ldots $$ where $$R_0=4(2^{1/d - 1} d^2)$$ for all $n \ge 1$ and likewise, we also have $f(\overline{\mathbb{D}}) \subset \overline{\mathbb{D}}_{R_0}$. 
 Since $P_n \to Q$ locally uniformly, we have 
\begin{equation}\label{lu}
\sup_{\overline{\mathbb{D}}_{R_0}}|P_n-Q|\to 0 
\end{equation}
as $n\to \infty$. Now consider any $w \in \partial \mathbb{D}$ and let $Q^{-1}(w)=\{x_1,...,x_s\}$. Choose $\tilde{\epsilon} < \epsilon$ so small that the closed disks $\overline{\Delta(x_j;\tilde{\epsilon})} ; j=1,...,s$ are pairwise disjoint. Then for $\delta = \delta(\epsilon,w) > 0$ chosen sufficiently small, $Q^{-1}(\overline{\Delta(w ; 2\delta)})$ is a compact subset of $\bigcup_{j=1}^{s} \Delta(x_{j} ; \frac{\tilde{\epsilon}}{2})$. For this $\delta$, let $z\in \Delta(w ; \delta) \cap \partial \mathbb{D}$. By (\ref{lu}), for all $n$ sufficiently large, say $n \geq N$, where $N$ is independent on $z$ but can depend on $w$ and $\delta$ we have $$|Q(f_n(z))-z|=|Q(f_n(z))-P_n(f_n(z))|<\delta$$
and hence $Q(f_n(z))\in \Delta(w ; 2\delta)$ so that $f_n(z)\in Q^{-1}(\overline{\Delta(w ; 2\delta)}) \subset \bigcup_{j=1}^{s} \Delta(x_{j} ; \frac{\tilde{\epsilon}}{2})$. Then we necessarily have $$f_n( \Delta(w ; \delta) \cap \partial \mathbb{D}) \subset \Delta(x_{j_n} ; \frac{\tilde{\epsilon}}{2}) $$ for some $j_{n} \in \{1,...,s \}$ as $$f_n(\Delta(w ; \delta) \cap \partial \mathbb{D}) $$ is connected. It follows that for all $n$ sufficiently large, we have $$|f_n(w_1) - f_n(w_2)| \leq \tilde{\epsilon} < \epsilon$$ for any $w_1,w_2 \in \Delta(w ; \delta) \cap \partial \mathbb{D}$. 

We now continue to prove statement (\ref{first}). For a fixed $w\in \partial\mathbb{D}$, let $[0,w]=\{tw:t \in [0,1]\}$ and $[0,w)=\{tw:t \in [0,1)\}$. By the same proof, we have 
\begin{enumerate}
\renewcommand{\labelenumi}{(\arabic{enumi})}
\setcounter{enumi}{2}
    \item \label{third} For every $w \in \partial \mathbb{D}$ and  $\epsilon > 0$, there exists a $\delta = \delta(\epsilon,w)$ such that for all $n$ sufficiently large, $$|\restr{f_n \ }{[0,w]}(w_1) - \restr{f_n \ }{[0,w]}(w_2) | < \epsilon $$ holds for every $w_1 , w_2 \in \Delta(w ; \delta) \cap [0,w]$
\end{enumerate}

Given $\epsilon>0$ and $w \in \partial \mathbb{D}$, choose $\delta$ in statement (\ref{third}) and by the continuity of $f$, pick $w'\in \Delta(w;\delta)\cap [0,w)$ so close to $w$ that $|f(w)-f(w')|<\epsilon$. Hence for all $n$ sufficiently large, $$|f_n(w)-f(w)|<|f_n(w)-f_n(w')|+|f_n(w')-f(w')|+|f(w')-f(w)|<3\epsilon$$
This implies that $$f_n(w) \rightarrow f(w)$$ holds for every $w \in  \partial \mathbb{D}$ which completes the proof of statement (\ref{first}). It follows readily from (\ref{first}) and (\ref{second}) that $$f_{n} \rightarrow f$$ holds uniformly on $\overline{\mathbb{D}}$. It remains to show that $Q$ is Schlicht normalized. To this end, let $b_{n} \in \partial D(P_n,1) \subset \overline{\mathbb{D}_{R_0}} ; n=1,2,...$ be a sequence of critical points, we may take a convergent subsequence $b_{n_k} \rightarrow b \in \overline{\mathbb{D}_{R_0}}$, such that $$b_{n_{k}} = f_{n_k}(\zeta_{k}) ; k=1,2, \dots$$ where $\zeta_{k} \in \partial \mathbb{D} ; k=1,2, \ldots$. By passing to a further subsequence (whose relabeling we suppress), we may assuming that $\zeta_{k} \rightarrow \zeta \in \partial \mathbb{D}$, and thus $$f_{n_{k}}(\zeta_{k}) \rightarrow f(\zeta)$$ by the uniform convergence $$f_{n} \rightarrow f$$ on $\overline{\mathbb{D}}$. Thus $b \in \partial D(Q,1)$ and by the local uniform convergence $P_{n} \rightarrow Q$ we conclude that $$Q'(b) = \lim_{k} P_{n_k}'(b_{n_{k}})  =0$$ and hence \begin{equation} \label{InequalityFort(Q)Second} t(Q) \leq 1 \end{equation} Combining (\ref{inequalityFort(Q)first}) and (\ref{InequalityFort(Q)Second}) we conclude that $t(Q) = 1$ so that $Q$ is Schlicht normalized. This concludes the proof of Lemma \ref{lem:10}. 

\end{proof}
\begin{remark} \label{compactnessOfSchlichtNormalizedPolynomials}
Since any sequence of Schlicht normalized polynomials $P_n \in \mathcal{P}_{d}$ satisfies $$P_n(D(P_n,1)) = \mathbb{D} ; n=1,2, \ldots $$ and since $$\mathbb{D}_{\frac{1}{4}} \subset D(P_n,1) ; n=1,2, \ldots$$ holds by the Koebe $\frac{1}{4}$-theorem, it follows from Montel's theorem that $$\restr{P_n \ }{\mathbb{D}_{\frac{1}{4}}} ; n=1,2, \ldots$$ forms a normal family, and any subsequential limit $Q$ of the above sequence must be a polynomial of degree $\leq d$ (as its Taylor coefficients of orders $> d$ vanish by local uniform convergence) which must be Schlicht normalized by Lemma \ref{lem:10}. This in particular gives that, for any $d \geq 2$, the class of Schlicht normalized polynomials of degree $\leq d$ is compact.
\end{remark}
We will also require the below continuity result for the proof of Theorem \ref{thm:3}.

\begin{corollary} \label{cor:4}
Let $P_n \in \mathcal{P}_{d}, n=1,2,3,...$ be a sequence of Schlicht normalized polynomials converging locally uniformly to the Schlicht normalized polynomial $Q \in \mathcal{P}_{d}$.

Then $$\mathrm{cap}(\overline{D(P_n,1)}) \rightarrow \mathrm{cap}(\overline{D(Q,1)})$$
\end{corollary}

\begin{proof}[Proof of Corollary \ref{cor:4}] Let $K_n=\overline{D(P_n,1)}=\overline{f_n(\mathbb{D})}$ and $K=\overline{D(Q,1)}=\overline{f(\mathbb{D})}$ and we know from the proof of Lemma \ref{lem:10} that $K_n,K \subset \mathbb{D}_{R_0}$. By Lemma \ref{lem:10}, $\|f_n-f\|_{\overline{\mathbb{D}}} \rightarrow 0$. We claim that the Hausdorff distance $d_H(K_n,K)\le \|f_n-f\|_{\overline{\mathbb{D}}} \rightarrow 0$ where $d_H(K_n,K)=\max\Bigl\{
\sup_{x\in K_n}\operatorname{dist}(x,K),\;
\sup_{y\in K}\operatorname{dist}(y,K_n)
\Bigr\}$. For any $x\in K_n$, there is $z\in\overline{\mathbb{D}}$ such that $x=f_n(z)$.  
Since $f(z)\in K$, we have 
$$\operatorname{dist}(x,K)\le\bigl|f_n(z)-f(z)\bigr|\le\|f_n-f\|_{\overline{\mathbb{D}}}$$
and hence
$$
\sup_{x\in K_n}\operatorname{dist}(x,K)\le\|f_n-f\|_{\overline{\mathbb{D}}}.
$$
Similarly, for any $y\in K$, we have
$\operatorname{dist}(y,K_n)\le\|f_n-f\|_{\overline{\mathbb{D}}}$ and therefore 
$d_H(K_n,K)\le \|f_n-f\|_{\overline{\mathbb{D}}} \rightarrow 0$.

Let $\Omega_n=\overline{\mathbb{C}}\backslash K_n$ and $\Omega=\overline{\mathbb{C}}\backslash K$. We now show that $\Omega_n\to\Omega$ in the sense of Carath\'eodory kernels with
respect to $\infty$, i.e. if $U$ is the Carath\'eodory kernel of $\Omega_n$, then $U=\Omega$.

Let $E$ be a compact subset of $\Omega$. Since $\Omega$ is connected and contains $\infty$,
there is a compact connected set
$$
L\subset\Omega,\qquad E\cup\{\infty\}\subset L.
$$
Because $L\cap K=\emptyset$ and $K\subset\overline{\mathbb D}_{R_0}$,
there is $\varepsilon>0$ such that
$$
\operatorname{dist}
 \bigl(L\cap\overline{\mathbb D}_{R_0+1},K\bigr)>\varepsilon.
$$
Recall that for all sufficiently large $n$,
$$
d_H(K_n,K)<\varepsilon,
$$
and hence $L\cap K_n=\emptyset$. In fact, if $z\in L\cap K_n$, then, since
$K_n\subset\overline{\mathbb D}_{R_0}$, we have
$$
z\in L\cap\overline{\mathbb D}_{R_0+1},
$$
and hence $\operatorname{dist}(z,K)>\varepsilon$. On the other hand,
the definition of the Hausdorff distance gives
$$
\operatorname{dist}(z,K)\le d_H(K_n,K)<\varepsilon 
$$
which is a contradiction. Therefore $L\cap K_n=\emptyset$. We then conclude that $E \subset L\subset \Omega_n=\overline{\mathbb{C}}\backslash K_n$ for all sufficiently large $n$. This argument applies equally to the subsequence $\Omega_{n_k}$
so that every compact subset of $\Omega$ is contained in $\Omega_{n_k}$. By the maximality of the kernel, this gives $\Omega\subset U$.

Now let $z\in U$.
Choose a compact arc $\gamma\subset U$ joining $z$ to $\infty$.
Since $\gamma$ is compact and $U$ is open, there exists
$\epsilon>0$ such that its closed spherical
$\epsilon$-neighborhood satisfies
$N_\epsilon(\gamma)\subset U$.
By the definition of the kernel, every compact subset of $U$ is
contained in $\Omega_{n_k}$ for all sufficiently large $k$. Hence
$$
N_\varepsilon(\gamma)\subset\Omega_{n_k}
$$
for all sufficiently large $k$. Therefore $\gamma$ stays a positive distance
from $K_{n_k}$. Since $K_{n_k}\to K$ in the Hausdorff metric, we obtain
$$
\gamma\cap K=\emptyset
$$
It follows that $\gamma \subset \Omega=\overline{\mathbb C}\setminus K$, and hence $z\in\Omega$. We conclude that $U=\Omega$.

We are now ready to apply the Carath\'eodory kernel theorem with the Riemann mappings $g_n : \overline{\mathbb{C}} \setminus \overline{\mathbb{D}} \rightarrow \overline{\mathbb{C}} \setminus \overline{D(P_n,1)}$ fixing the point at $\infty$ and satisfying $$g_n(z) = \mathrm{cap} (\overline{ D(P_n,1)}) z  + a_{0,n} +  O(z^{-1})$$ near $\infty$. Since we have just shown that $\Omega_n=\overline{\mathbb{C}} \setminus \overline{D(P_n,1)} ; n=1,2, \ldots$ converges in the sense of Carath\'eodory kernels with
respect to $\infty$ to $\Omega=\overline{\mathbb{C}} \setminus \overline{D(Q,1)}$, the Carath\'eodory kernel theorem implies that  $g_n $ converges locally uniformly (with respect to the spherical metric on $\overline{\mathbb{C}}$) to the Riemann map $g : \overline{\mathbb{C}} \setminus \overline{\mathbb{D}} \rightarrow \overline{\mathbb{C}} \setminus \overline{D(Q,1)}$, which satisfies $$g(z) = \mathrm{cap}(\overline{D(Q,1)}) z   + a_0+ O(z^{-1})$$ near $\infty$.

 We conclude that $\mathrm{cap}(\overline{D(P_n,1)}) \rightarrow \mathrm{cap}(\overline{D(Q,1)})$ as desired.
\end{proof}

\begin{proof}[Proof of Theorem \ref{thm:3}]

Since the quantity we are maximizing over $\mathcal{P}_{d}$ $$\frac{||P'||_{\overline{D(P,t(P))}} \mathrm{cap}(\overline{D(P,t(P))})}{t(P)}$$ is invariant under linear rescalings of the form $P \rightarrow \alpha P(\beta z)$, we may restrict consideration to polynomials satisfying $t(P) = 1$ and $P'(0) = 1$, i.e. we restrict consideration to Schlicht normalized polynomials of degree $\leq d$. Let $$P_n \in \mathcal{P}_{d} ; n=1,2,...$$ be an extremizing sequence of Schlicht normalized polynomials, i.e.  $$\frac{||P_n'||_{\overline{D(P_n,1)}} \mathrm{cap}(\overline{D(P_n,1)})}{t(P_n)} \rightarrow_{n \rightarrow \infty} \sup_{P \in \mathcal{P}_{d}} \frac{||P'||_{\overline{D(P,t(P))}} \mathrm{cap}(\overline{D(P,t(P))})}{t(P)} $$ By the compactness of Schlicht normalized polynomials of degrees $\leq d$ (see Remark \ref{compactnessOfSchlichtNormalizedPolynomials}) we may without loss of generality, after passing to an appropriate subsequence whose relabeling we suppress, assume that there exists a polynomial $Q$ for which $P_n \rightarrow Q$ holds locally uniformly. Then (by Lemma \ref{lem:10}) $Q$ is also Schlicht normalized of degree $\leq d$, and by Corollary \ref{cor:4} $$ \mathrm{cap}(\overline{D(P_n,1)}) \rightarrow  \mathrm{cap}(\overline{D(Q,1)})$$ We claim that $$||P_n'||_{ \overline{D(P_n,1)}}  \rightarrow ||Q'||_{\overline{D(Q,1)}} $$ Indeed, this follows from  the fact that $\overline{D(P_n,1)} \rightarrow \overline{D(Q,1)}$ in the Hausdorff metric (which is a consequence of Lemma \ref{lem:10}), and the local uniform convergence $P_{n} \rightarrow Q$. It follows that $$||P_{n} '||_{\overline{D(P_n,1)}} \mathrm{cap}(\overline{D(P_n,1)}) \rightarrow_{n \rightarrow \infty} ||Q'||_{\overline{D(Q,1)}} \mathrm{cap}(\overline{D(Q,1)})$$ so that $$||Q'||_{\overline{D(Q,1)}} \mathrm{cap}(\overline{D(Q,1)}) = \sup_{P \in \mathcal{P}_{d}} \frac{||P'||_{\overline{D(P,t(P))}} \mathrm{cap}(\overline{D(P,t(P))})}{t(P)}$$ In other words, $Q \in \mathcal{P}_{d}$ is a (Schlicht normalized) extremal polynomial for this quantity.

Using Eremenko and Hayman's quasiconformal deformation argument (as it is applied in \cite{Eremenko2007}), one then obtains that we may take $Q$ to satisfy $|Q(b)| = t(Q) = 1$ for every critical point $b \in Z(Q') \setminus Z(Q)$. A full justification for this is as follows. In the notation of the proof of Theorem \ref{thm:9} (i.e setting $P \coloneq Q$ to be an extremal polynomial), we have by the maximum modulus principle \begin{equation} \label{maxMod} \log ||P_{\lambda}'||_{\overline{D(P_{\lambda},1)}} = \log||P_{\lambda}'||_{\partial D(P_{\lambda},1)} \end{equation} and since $\phi_{\lambda}$ is holomorphic in some neighbourhood of $\partial D(P,1)$, and $$\partial D(P_{\lambda},1) = \phi_{\lambda}(\partial D(P,1))$$ we are guaranteed subharmonicity of $$\lambda \rightarrow \log ||P_{\lambda}'||_{\partial D(P_{\lambda},1)}$$ in some neighbourhood of the origin, in the same way as in the proof of Lemma 2 of \cite{Eremenko2007}. It follows by the above subharmonicity, and the equality (\ref{maxMod}), that $$\lambda \rightarrow \log ||P_{\lambda}'||_{\overline{D(P_{\lambda},1)}}$$ is subharmonic in a neighbourhood of the origin. Similarly, in view of the fact (since these are Jordan domains) that \begin{equation} \label{capEqualsBoundary} \mathrm{cap}(\overline{D(P_{\lambda},1)}) = \mathrm{cap}(\partial D(P_{\lambda},1)) \end{equation} and the subharmonicity of $$\lambda \rightarrow \log \mathrm{cap}(\partial D(P_{\lambda},1))$$ which follows in the same way as in the proof of Lemma 3 of \cite{Eremenko2007}, we find that $$\lambda \rightarrow \log \mathrm{cap}(\overline{D(P_{\lambda},1)}) $$ is subharmonic in a neighbourhood of the origin. Thus, the logarithm of the expression we seek to maximize is subharmonic in the parameter $\lambda$ and must be constant. We are thus free to move all (non-zero) critical values of $P$ not lying on $\partial \mathbb{D}$ towards $\partial \mathbb{D}$, and recover an extremal polynomial all of whose non-zero critical values lie on the unit circle via a limiting argument as in \cite{eremenko1999length} and \cite{Eremenko2007}. 

Let $Q$ denote such an extremal polynomial (with all non-zero critical values of equal modulus $1 = t(Q)$). The condition on $Q$'s (non-zero) critical values implies $Q^{-1}(\overline{\mathbb{D}_{t(Q)}})$ is connected (see e.g. the proof of Lemma 6 in \cite{eremenko1999length}) and hence $$Q^{-1}(\overline{\mathbb{D}_{t(Q)}})= E(Q,t(Q))$$ where here (see Section \ref{Notation}) for a polynomial $H$ fixing the origin, we let $E(H,r)$ denote the component of $H^{-1}(\overline{\mathbb{D}}_{r})$ containing the origin. Let $m \coloneq \deg(Q)$. We have already established that $m \geq 2$, and by construction we have $d \geq m$. Theorem \ref{thm:5} gives

\begin{equation} \label{CapacCompApp}
\frac{\mathrm{cap}(E(Q,t(Q)))}{\mathrm{cap}(\overline{D(Q,t(Q))})} = \frac{\mathrm{cap}({Q^{-1}(\overline{\mathbb{D}_{t(Q)}})})}{\mathrm{cap}(\overline{D(Q,t(Q))})} \geq C(m) = (2m-1)^{\frac{1}{m}} \geq (2d-1)^{\frac{1}{d}}
\end{equation}

The final inequality above follows from the fact that $$x \rightarrow (2x-1)^{\frac{1}{x}}$$ is strictly decreasing for $x \in \{2,3,4,5,... \}$ 

The above inequality (\ref{CapacCompApp}) combined with Eremenko's Markov Inequality (\cite{Eremenko2007}, Theorem 1) gives

$$||Q'||_{ \overline{D(Q,t(Q))}} \mathrm{cap}(\overline{D(Q,t(Q))}) \leq C(m)^{-1}||Q'||_{E(Q,t(Q))} \mathrm{cap}({E(Q,t(Q))}) $$ $$\leq C(m)^{-1} 2^{\frac{1}{m}-1} m^2 t(Q)   \leq C(m)^{-1} 2^{\frac{1}{d}-1} d^2 t(Q)  $$ $$= 2^{\frac{1}{d}-1} C(m)^{-1} d^2 \leq (\frac{2}{2d-1})^{\frac{1}{d}} \frac{d^2}{2}$$

This completes the proof of Theorem \ref{thm:3}.

\end{proof}

\subsection{Proof of Corollary \ref{cor:2}} \label{sec4.1}
As a corollary of Theorem \ref{thm:3}, we obtain.

\DMVC*

\begin{proof}[Proof of Corollary \ref{cor:2}]

The existence of an extremal polynomial follows along the same lines as Theorem \ref{thm:9}. Namely, we use the fact that the class of Schlicht normalized polynomials is compact (see Remark \ref{compactnessOfSchlichtNormalizedPolynomials}), and any critical point sequence $b_n$ lying on the boundary of $D(P_n,1) = D(P_n,t(P_n))$, where $P_n$ is a Schlicht normalized (i.e $t(P_n) =1$ and $P_n'(0)=1$ for each $n$) sequence of polynomials converging locally uniformly to a Schlicht normalized polynomial $P$, converges to a critical point $b$ lying on the boundary of $D(P,1)$ after passing to a further subsequence if necessary (we suppress the relabeling). This follows from the fact that $b_n$ is bounded (away from $0$ and $\infty$) respectively by the Koebe $\frac{1}{4}$-Theorem and (\ref{Markov}). To prove that such a (Schlicht normalized) extremal polynomial may be chosen so that all its non-zero critical values lie on the unit circle, one employs essentially the same argument as in the proof of Theorem \ref{thm:9}, since the problem of minimizing $L(P)$ over Schlicht normalized polynomials $P \in \mathcal{P}_{d}$ is the same as the problem of maximizing the quantity $\max_{b \in Z(P') \cap \partial D(P,1)} |b|$ over all Schlicht normalized polynomials $P \in \mathcal{P}_{d}$, and the latter quantity is subharmonic with respect to the quasiconformal deformation used in the proof of Theorem \ref{thm:9}. More in detail, we have that the deformed polynomials $$P_{\lambda} : \lambda \in N$$ are Schlicht normalized, where we are using the notation of the proof of Theorem \ref{thm:9}, and $$\phi_{\lambda}(\partial D(P,1)) = \partial D(P_{\lambda},1) ; \lambda \in N$$ $$Z(P_{\lambda}') \cap \partial D(P_{\lambda},1) = \phi_{\lambda}(Z(P') \cap \partial D(P,1)) ; \lambda \in N$$ and hence $$\lambda \rightarrow \max_{b_{\lambda} \in Z(P_{\lambda}') \cap \partial D(P_{\lambda},1)} |b_{\lambda}| = \max_{b \in Z(P') \cap \partial D(P,1)} |\phi_{\lambda}(b)| $$ is the maximum of finitely many subharmonic functions and is thus subharmonic in $\lambda \in N$. Since it attains a maximum at the interior point $0 \in N$, we conclude from the maximum principle for subharmonic functions that it is constant, and we can proceed to obtain a (Schlicht normalized) extremal polynomial of degree $\leq d$ with all non-zero critical values lying on the unit circle in the same way as in the proof of Theorem \ref{thm:9}.

Let $H$ be \textbf{any} polynomial in $\mathcal{P}_{d}$. To prove the first part of the above theorem we use the following argument. 

Fix any $b \in Z(H') \cap \partial D(H,t(H))$ and rescale $$H \rightarrow \alpha H(\beta z) \eqcolon P(z)$$ So that $b=1,P(1) = 1$ and $t(P) = 1$. Then $b=1$ is a critical point of $P$ which is the terminal point of a Jordan arc $\gamma$ from $0$ to $1$ (contained in $D(P,t(P)) = D(P,1)$) which $P$ maps homeomorphically onto $[0,1]$. Let $w_0$ be the unique point on $\gamma$ such that $P(w_0) = \frac{1}{2}$. Consider $Q(z) = 2P(z + w_0) - 1$.

Let $D(Q,1)$ denote the component of $Q^{-1}(\mathbb{D})$ containing zero. Then the (open) Jordan arc $\gamma - w_0$ with endpoints $-w_0$ and $1 - w_0$ contains $0$ and is mapped by $Q$ homeomorphically onto $(-1,1)$. It follows that $\gamma - w_0 \subset D(Q,1)$. Let $f \coloneq \restr{P^{-1}}{\mathbb{D}}$, where $P^{-1}$ denotes the branch defined on the unit disk which fixes the origin. We in fact have that $t(Q) = 1$, since $w_0 = f(\frac{1}{2})$, and $Q$ maps the set $f(\Delta(\frac{1}{2};\frac{1}{2})) - w_0$ (which contains $0$) univalently onto the open unit disk, so that $t(Q) \geq 1$. On the other hand the point $1 - w_0$ is a critical point of $Q$ lying on the boundary of this set, thus we necessarily have that $t(Q) = 1$. By Theorem \ref{thm:3} and monotonicity of logarithmic capacity we have:

$$\mathrm{cap} (\gamma - w_0) 2 |P'(0)| \leq \mathrm{cap}(\overline{D(Q,1)}) 2 |P'(0)| = \mathrm{cap}(\overline{D(Q,1)}) |Q'(-w_0)| $$ $$ \leq \mathrm{cap} (\overline{D(Q,1)}) \sup_{z \in \overline{D(Q,1)}} |Q'(z)| \leq 2^{\frac{1}{d} - 1} d^2 (2d-1)^{-\frac{1}{d}}$$

Rearranging the above while noting that $\mathrm{cap}(\gamma - w_0) = \mathrm{cap}(\gamma)$ gives.

 \begin{equation} \label{finalIneqCor2} \frac{1}{|P'(0)|} \geq \frac{(2d-1)^{\frac{1}{d}}2^{2-\frac{1}{d}} \mathrm{cap}(\gamma)}{d^2} \end{equation}

 Using that $\mathrm{cap}(\gamma) \geq \frac{1}{4}$ (since this is a continuum containing the points $0$ and $1$), we obtain

 $$ S(H,b) = S(P,1) = \frac{1}{|P'(0)|} \geq \frac{(2d-1)^{\frac{1}{d}} 2^{-\frac{1}{d}} }{d^2} = \frac{(d - \frac{1}{2})^{\frac{1}{d}}}{d^2}$$

This concludes the proof of Corollary \ref{cor:2}.

\end{proof}

\subsection{Sharpness of Corollary \ref{cor:2}} \label{sec4.2}
The right-hand side of (\ref{finalIneqCor2}) is exactly equal to $$\frac{\mathrm{cap}(\gamma)}{\frac{1}{4}} \frac{\left(d-\frac{1}{2}\right)^{\frac{1}{d}}}{d^2} = \frac{\mathrm{cap}(\gamma)}{\mathrm{cap}[0,1]} \frac{\left(d-\frac{1}{2}\right)^{\frac{1}{d}}}{d^2} =  \frac{\mathrm{cap}(\gamma)}{\mathrm{cap}[0,b]} \frac{\left(d-\frac{1}{2}\right)^{\frac{1}{d}}}{d^2} $$ Where $b = 1$ is the endpoint of the Jordan arc $\gamma \subset \overline{D(P,1)}$ from $0$ to the critical point $1$ such that $P(\gamma) = [0,1] = [0,P(1)] = [0,P(b)]$. Note that $$ \frac{\mathrm{cap}(\gamma)}{\mathrm{cap}[0,b]}$$ is invariant under rescalings of the form $P \rightarrow \alpha P(\beta z)$. Therefore we obtain that for any polynomial $P \in \mathcal{P}_{d} ; d \geq 2$ and $$b \in Z(P') \cap \partial D(P,t(P))$$ letting $\gamma$ denote the Jordan arc from $0$ to $b$ in $\overline{ D(P,t(P))}$ for which $P(\gamma) = [0, P(b)]$, the below inequality holds
$$
S(P,b) \geq \frac{\mathrm{cap}(\gamma)}{\mathrm{cap}[0,b]}\frac{\left(d-\frac{1}{2}\right)^{\frac{1}{d}}}{d^2}
\geq \frac{\left(d - \frac{1}{2}\right)^{\frac{1}{d}}}{d^2}.
$$
Note that $$\frac{\mathrm{cap}(\gamma)}{\mathrm{cap}[0,b]} \geq 1$$ and equality holds if and only if $\gamma$ coincides with the segment $[0,b]$ (this follows from Lemma A4 on page 307 of \cite{dubinin2014} since $\gamma$ is a Jordan arc and is thus equal to the support of its equilibrium measure). 

We state this as Corollary \ref{SharperDualIneq} below.

\SharperDMVC*

One consequence of Corollary \ref{SharperDualIneq} (and a version of the Markov inequality for so-called \textbf{monotone polynomials}, see (\ref{MonotoneMarkov}) below) is that we in fact necessarily have $$L(d) \coloneq \inf_{P \in \mathcal{P}_{d}} L(P) >\frac{\left(d - \frac{1}{2}\right)^{\frac{1}{d}}}{d^2} $$ for all $d \geq 2$. We formulate this as Theorem \ref{Nonsharpness}.

\begin{theorem} \label{Nonsharpness}
For any $d \geq 2$ let $$L(d) \coloneq \inf_{P \in \mathcal{P}_{d}} L(P)$$ The below inequality holds $$L(d) > \frac{\left(d - \frac{1}{2}\right)^{\frac{1}{d}}}{d^2} = \frac{2^{-\frac{1}{d}} C(d)}{d^2} $$
\end{theorem}

\begin{proof}
    Suppose, for the sake of a contradiction, that $$L(d) = L(P_0)  = \frac{\left(\frac{2d-1}{2}\right)^{\frac{1}{d}}}{d^2} $$ where $P_0 \in \mathcal{P}_{d}$ is the extremal polynomial of Corollary \ref{cor:2}. We may assume without loss of generality that $1 \in \partial D(P_0,t(P_0)) \cap Z(P_0')$, $P_0(1) = 1 = t(P_0)$ and $S(P_0,1) = L(P_0)$. Then, by Corollary \ref{SharperDualIneq} $$\frac{\mathrm{cap}(\gamma)}{\mathrm{cap}[0,1]} = 1$$ so that $$\gamma = [0,1]$$ $P_0$ is \textbf{monotone} on $[0,1]$ and is necessarily real. For monotone real polynomials, $Q$ of degree $n \geq 1$ on $[-1,1]$ one has the below version of the Markov inequality (see page 3 of \cite{BorweinErdelyi1997})  \begin{equation}\label{MonotoneMarkov} 
\sup_{z\in [-1,1]}|Q'(z)| \;\le\; \frac{(n+1)^2}{4}\ \sup_{z\in [-1,1]}|Q(z)|,
\end{equation} 

Applying (\ref{MonotoneMarkov}) to $$Q_0(z) \coloneq 2 P_0(\frac{1}{2} (z+1)) - 1$$ we obtain that $$|P_0'(0)| \leq \frac{(m+1)^2}{4}$$ where $m \coloneq \deg(P_0) \leq d$, so that $$\frac{\left(\frac{2d-1}{2}\right)^{\frac{1}{d}}}{d^2} = L(d) = L(P_0) = \frac{1}{|P_0'(0)|} \geq \frac{4}{(m+1)^2} \geq \frac{4}{(d+1)^2} $$

This is a contradiction, as for every $d \geq 2$, we have $$\frac{\left(\frac{2d-1}{2}\right)^{\frac{1}{d}}}{d^2} <\frac{4}{(d+1)^2} $$

\end{proof}

\begin{remark} \label{asymptoticFactorInDualConjecture}
The proof of Theorem \ref{Nonsharpness} suggests one might expect that $$L(d) \geq \frac{c_0}{d^2}$$ as $d \rightarrow \infty$, for some absolute constant $c_0 > 1$. Indeed, for extremal polynomials $P$ with a capacity ratio close to $1$, the arc $\gamma$ is close (in the Hausdorff metric) to a straight line segment with the same endpoints, meaning $P$ is (asymptotically) ``close to monotone''. If one could establish an asymptotic version of the Markov inequality (\ref{MonotoneMarkov}) for monotone polynomials, the argument in Theorem \ref{Nonsharpness} would immediately yield $c_0 > 1$.
\end{remark}

\medskip
We end this section by noting that in contrast to Corollary \ref{cor:2},  Dubinin \cite{Dubinin2019} established the following \textbf{asymptotically sharp} inequality for the ``outermost'' critical lemniscates.

\begin{theorem}[Dubinin; Theorem 1 of \cite{Dubinin2019}]
    Let $P \in \mathcal{P}_{d}$, and let $$R(P) \coloneq \max_{b \in Z(P')} |P(b)|$$ Then for any $$b \in \partial E(P,R(P)) \cap Z(P')$$ the below inequality holds

    $$S(P,b) \geq \frac{ \tan{\frac{\pi}{4d}}}{d} \sim \frac{\frac{\pi}{4}}{ d^2}$$
\end{theorem}

    The above inequality is \textbf{asymptotically sharp}, in the sense that there exists a sequence of polynomials $P_d \in \mathcal{P}_{d}$, and critical points $b_d \in \partial E(P_d,R(P_d))$, such that $$\frac{S(P_d,b_d)}{\frac{ \tan{\frac{\pi}{4d}}}{d}} \rightarrow_{d \rightarrow \infty} 1$$

    For the asymptotic extremality, one can take $$P_d = T_d(z + \cos(\frac{\pi}{2d}))$$ and $$b_d \coloneq -\cos{\frac{\pi}{d}} - \cos{\frac{\pi}{2d} }$$ where $T_d$ is the degree $d$ Chebyshev polynomial of the first kind.

\bibliographystyle{spmpsci} 
\bibliography{sn-bibliography}

\end{document}